 \documentclass[11pt]{article}
 \usepackage{fullpage}

\usepackage{amssymb,amsmath,amsfonts}
\usepackage{amsmath}
\usepackage{graphicx}
\usepackage{times}
 \usepackage{ifpdf}
 \ifpdf\fi
\usepackage{float}

\usepackage{algorithm}
\usepackage{algorithmicx}
\usepackage[noend]{algpseudocode}

\usepackage{epsfig}
 \usepackage{times}

\newcommand{\ignore}[1]{}
\newcommand{\signore}[1]{}

\newcommand{\notinproc}[1]{}

\newtheorem{thm}{Theorem}[section]
\newtheorem{theorem}{Theorem}[section]

\newtheorem{lemma}[thm]{Lemma}

\newtheorem{definition}[thm]{Definition}

\newfloat{fexample}{thp}{lop}
\floatname{fexample}{Example}

 \newcommand{\qed}{\hfill \rule{1ex}{1ex}\medskip\\}
 \newenvironment{proof}{\paragraph{Proof}}{\qed}

\def\range{\mbox{{\sc rg}}}

\def\cl{\mbox{cl}}

\def\var{\mbox{\sc var}}

\def\E{{\textsf E}}

\def\vecv{\boldsymbol{v}}
\def\vecz{\boldsymbol{z}}
\def\vecw{\boldsymbol{w}}
\def\vecy{\boldsymbol{y}}

\def\p{p}
\def\L{L*}
\def\U{U*}
\def\1DE{Monotone Estimation}

\begin{document}

 \title{Estimation for Monotone Sampling:\\
  Competitiveness and Customization}

\date{}

 \ignore{
\numberofauthors{1}
\author{
\alignauthor Edith Cohen\\
       \affaddr{Microsoft Research SVC}\\
       \affaddr{Mountain View, CA, USA}\\
       \email{editco@microsoft.com}
}
 }

\author{
Edith Cohen\thanks{Microsoft Research, Mountain View, CA, USA
{\tt editco@microsoft.com}} $^\dagger$ 
}



 \maketitle
\begin{abstract}
 { \small Random samples are lossy summaries which allow queries posed
   over the data to be approximated by applying an appropriate
   estimator to the sample.  The effectiveness of sampling, however,
   hinges on estimator selection.  The choice of estimators is
   subjected to global requirements, such as unbiasedness and range
   restrictions on the estimate value, and ideally, we seek estimators
   that are both efficient to derive and apply and {\em admissible}
   (not dominated, in terms of variance, by other estimators).
   Nevertheless, for a given data domain, sampling scheme, and query,
   there are many admissible estimators.

   We study the choice of admissible nonnegative and unbiased
   estimators for monotone sampling schemes.  Monotone sampling
   schemes are implicit in many applications of massive data set
   analysis.  
Our main contribution is general derivations of admissible
estimators with desirable properties.  We present a construction of {\em order-optimal}
estimators, which minimize variance according to {\em any} specified
priorities over the data domain.  Order-optimality allows us to
customize the derivation to common patterns that we can learn or
observe in the data.  When we prioritize 
lower values (e.g., more similar data sets when estimating
difference), we  obtain the L$^*$ estimator, which is the unique
monotone admissible estimator. We show that the L$^*$ estimator is
4-competitive and dominates the classic Horvitz-Thompson estimator.
These properties make the L$^*$ estimator a natural default choice.
We also present the U$^*$ estimator, which prioritizes large values
(e.g., less similar data sets).  Our estimator constructions are both
easy to apply and possess desirable properties, allowing us to make
the most from our summarized data.

 }
\ignore{
We build on a notion of 
order optimality, defined with respect to an order on 
data vectors:  any estimator with lower variance on data must have
higher variance on some preceding data.
  By choosing the order so that more prevalent data patterns are prioritized, we
can tailor our estimators to the data.  

We define a range of optimal estimates.  The lower point in that
range, which we name the \L\ estimator, is both 4-competitive and
variance optimal, in that the second moment on any data and any
function is within a factor of 4 of the minimum possible for the data.
The \L\ estimator dominates the Horvitz-Thompson estimator and is order
optimal with respect to the order that prioritizes data with lower $f$
value.  The upper point in the range, which is the \U\ estimator, is
order optimal with respect to the reverse order, under some conditions
on $f$.  More generally, we show how to compute, for any function and
order for which an unbiased and nonnegative estimator exists, an
order-optimal estimator.  }
 \end{abstract}


\section{Introduction} \label{intro:sec}

Random sampling is a common tool in the analysis of massive
data.  Sampling is highly suitable for parallel or distributed
platforms.  The samples facilitate scalable approximate processing of queries posed
over the original data, when exact processing is too resource
consuming or when the original data is no longer available.  Random
samples have a distinct advantage over other synopsis in their
flexibility.   In particular, they naturally support {\em domain} (subset)
queries, which specify a selected set of records.
Moreover, the same sample can be used for 
basic statistics,
such as sums, moments, and averages, and more complex relations:
distinct counts, size of set intersections, and difference
norms.  

The value of a sample hinges on the accuracy within which we
can estimate query results.  In turn, this boils down to the {\em
  estimators} we use, which are the functions we apply to the sample
to produce the estimate.  As a rule, we are interested in
estimators that satisfy desirable {\em global} properties,  which must hold
for {\em all possible data} in our data domain.  Common desirable
properties are:
\begin{trivlist}
\item $\bullet$
    {\em Unbiasedness}, which means that the expectation of the estimate is 
    equal to the estimated value.  Unbiasedness is particularly important 
    when we are ultimately interested in estimating a sum aggregate,
    and our estimator is  applied to each summand. 
    Typically, the estimate for each summand has high variance, but
    with unbiasedness (and pairwise independence), the
relative error to decreases with  aggregation. 
\item $\bullet$
   {\em Range restriction} of estimates:  since the estimate is often used as a substitute of
   the true value, we would like it to be from the same domain as the
 query result.  Often, the domain is {\em nonnegative}  and we would
 like the estimate to be nonnegative as well.  Another natural
 restriction is {\em boundedness} which means that all the estimate
 for each given input is bounded.
\item $\bullet$
    {\em Finite variance} (implied by boundedness but less restrictive)
\end{trivlist}

\ignore{
 A natural  fundamental question is deriving estimators which minimize
the variance for a particular data.
For a function $f$ and data $\vecv$, the trivial
estimator which outputs the fixed value equal to $f(\vecv)$ 
on all outcomes consistent with data $\vecv$ has
zero variance for $\vecv$, but might violate global
requirements.  
Over estimators that are unbiased and nonnegative, 
the minimum possible variance for
particular data can be strictly positive: If our data domain is
$\vecv\in [0,1]$, then when the value is not sampled we must output an
estimate of $0$, to allow for nonnegative unbiased estimates when the
data value is $0$.  This means that whenever the sampling probability
is strictly less than $1$, variance is positive on all positive data
values.  
}

Perhaps the most ubiquitous quality measure of an estimator is its
variance.  The variance, however, is a function of the input data.  An
important concept in estimation theory is a
Uniform Minimum Variance Unbiased (UMVUE) 
estimator \cite{surveysampling2:book}, that is, a single estimator which attains the minimum 
possible variance  for all inputs in our data domain 
\cite{Lanke:Metrica1973}.  
A UMVUE estimator, however, generally does not exist.
We instead seek an
{\em admissible} (Pareto variance optimal) estimator \cite{surveysampling2:book} -- meaning  that strict improvement 
 is not possible without violating some global properties.  More
precisely,  an estimator is admissible if there is no other estimator that satisfies the global properties
with at most the variance
of our estimator on all data and strictly lower 
variance on some data.   A UMVUE  must be admissible, but
when one does not exist, there is typically a full Pareto front of
admissible estimators.
We recently proposed 
 {\em variance competitiveness} ~\cite{CKsharedseed:2012}, as a 
robust ``worst-case'' performance measure when there is no UMVUE.
We defined the variance competitive ratio to be 
the maximum, over data, of the ratio of the expectation of the square of our estimator
to the minimum possible for the data subject to the global properties.
A small ratio
 means that variance on each input in the data domain is not
too far off the minimum variance attainable on this data by an
estimator which satisfies the global properties.

\ignore{
 and constitutes
a robust ``worst case''  measure, which captures what we can
guarantee for all possible data.
This notion bridge the
  gap between 
the classic notion of  UMVUE (uniform minimum
  variance unbiased) estimators \cite{surveysampling2:book} , which attain the minimum possible variance for
  all data (and therefore are also variance-optimal), but generally do
  not exist \cite{Lanke:Metrica1973}, 

and the practice of estimator selections with no
  ``worst-case''  guarantees.
}

  We work with the following definition of a sampling scheme.  In the
  sequel we show how it applies to common sampling schemes and their applications.
\begin{quote}
A {\em monotone sampling scheme} $({\bf V},S^*)$ is specified by a 
{\em data domain} ${\bf V}$
and a mapping 
$S^*: {\bf V}\times (0,1] \rightarrow 2^{{\bf V}}$. The mapping is
such that the set $S^*(\vecv,u)$ for a fixed $v$ is monotone
non-decreasing with $u$.
\end{quote}
The sampling interpretation is that a {\em sample} 
$S(\vecv,u)$ of the {\em input}
$\vecv$ (which we also refer to as the {\em data vector})  is obtained by drawing a {\em seed} $u \sim U[0,1]$, uniformly at random 
from $[0,1]$.  The sample deterministically depends on $\vecv$ and the (random) {\em seed} $u$.  
The mapping $S^*(\vecv,u)$ is the set of 
all data vectors that are consistent with $S$ (which we assume
includes the seed value $u$).  It represents all the information we
can glean from the sample on the input.  In particular, we must have $\vecv\in S^*(\vecv,u)$ for all
$\vecv$ and $u$.  The sampling scheme is monotone in the
randomization: When fixing $\vecv$, the set $S^*(\vecv,u)$ is 
non-decreasing with $u$, that is, the smaller $u$ is, the more 
information we have on the data $\vecv$.  

In the applications we consider, the (expected) representation size
of the sample $S(\vecv,u)$ is typically much smaller size than
$\vecv$. 
The set $S^*$ can be very large (or infinite), and our estimators will
only depend on performing certain operations on it, such as obtaining
the infimum of some function.
Monotone sampling can also be interpreted as obtaining a ``measurement''  $S(\vecv,u)$ 
of the data $\vecv$, where $u$ determines the granularity of 
our measuring instrument. 
Ultimately, the goal is to recover some function of the data from the
sample (the outcome of our measurement):

 \begin{quote}
A {\em monotone estimation} problem is specified by a monotone
sampling scheme and a nonnegative function 
$f: {\bf  V} \geq 0$.  The goal is to specify an {\em estimator},
which is a function of all possible outcomes
$\hat{f}:{\cal S} \geq
0$, where ${\cal S} = \{S(\vecv,u) | \vecv \in {\bf V}, u\in (0,1]$.
The estimator should be unbiased
$\forall \vecv,\ \E_{u\sim U[0,1]} \hat{f}(S(\vecv,u)) = f(\vecv)$ and
satisfy some other desirable properties.
  \end{quote}
 The interpretation is that we obtain a query, specified in the form of a nonnegative function 
$f: {\bf  V} \geq 0$ on all possible data vectors $\vecv$.  We are
interested in knowing $f(\vecv)$, but we can not see $\vecv$ and
only have access to the sample $S$. The sample provides us with little
information on $\vecv$, and thus on $f(\vecv)$.  We approximate
$f(\vecv)$ by applying  an {\em estimator}, $\hat{f}(S)\geq 0$ to the sample.
The monotone estimation problem is a bundling of a function $f$ and a monotone sampling
scheme.  We are interested in estimators $\hat{f}$ that satisfy
properties. We always require nonnegativity and
unbiasedness and consider admissibilitiy, variance
competitiveness, and what we call customization (lower variance on some data patterns).


Our formulation departs from traditional estimation theory.  We view the data
vectors in the domain as the possible inputs to the
sampling scheme, and we treat estimator derivation as an
optimization problem.  The variance of the estimator parallels the
``performance'' we obtain on a certain input.  
The work horse of estimation theory, the maximum likelihood
estimator, is not even applicable here as it does not distiguish
between the different data vectors in $S^*$.
Instead,  the random
``coin flips,'' in the form of the seed $u$, that are available to the
estimator are used to restrict the set $S^*$ and obtain meaningful estimates.

  We next show how monotone sampling relates to the well-studied model of
coordinated sampling, that has extensive applications in massive data
analysis.  In particular, estimator constructions for monotone
estimation can be applied to estimate functions over coordinated samples.


\subsection*{Coordinated shared-seed sampling}
In this framework our data has a matrix form of 
multiple {\em instances} ($r>1$), where each 
instance (row) has the form of a weight assignment to the (same) set of 
items (columns). 
Different  instances may correspond to 
snapshots, activity logs, measurements,
or repeated surveys that are taken at different times or locations.  
When instances correspond to documents, items can correspond 
to features.  When  instances are network neighborhoods, items can 
correspond  to members or objects they store.   


Over such data, we are interested in
queries which depend on two or more instances and a subset (or all)
items.  Some examples are Jaccard similarity, distance norms, or the
number of distinct items with positive entry in at least one instance
(distinct count). These queries can be conditioned on a subset of items.

 Such queries often can be expressed, or can be well approximated, by a sum over
(selected) items of an {\em item function} that is applied to the tuple
containing the values of the item in the different instances. 
Distinct count is a sum aggregate of
logical OR and the $L_p$ difference  is the 
$p$th root of $L_p^p$, which sum-aggregates $|v_1-v_2|^p$, when $r=2$.
For $r\geq 2$ instances, we can consider sum aggregates of the exponentiated range functions
$\range_p(\vecv)=(\max(\vecv)-\min(\vecv))^p$, where
$p>0$.
This is made concrete in Example \ref{example1} which illustrates a data set
of 3 instances over 8 items, example  
queries, specified over a selected set of items, and the corresponding item
functions.

We now assume that 
each instance is sampled and the 
sample of each instance contains a subset of the
items that were active in the instance (had a positive weight).
Common sampling schemes for a single instance are
Probability Proportional to Size (PPS) ~\cite{Hajekbook1981} or bottom-$k$ sampling which
includes
Reservoir sampling~\cite{Knuth2f,Vit85}, Priority (Sequential Poisson)~\cite{Ohlsson_SPS:1998,DLT:jacm07},
or Successive weighted sample
without
replacement~\cite{Rosen1972:successive,ES:IPL2006,bottomk07:ds}.  
The sampling of items in each instance can be completely independent
or slightly dependent (as with Reservoir or bottom-$k$ sampling, which
samples exactly $k$ items).

  Coordinated sampling is a way of specifying the randomization so
  that the sampling of different instances utilizes the same
  ``randomization''~\cite{BrEaJo:1972,Saavedra:1995,ECohen6f,Ohlsson:2000,Rosen1997a,Broder:CPM00,BRODER:sequences97,CK:sigmetrics09,LiChurchHastie:NIPS2008,multiw:VLDB2009}.
  That is, the sampling of the same item in different instances
becomes very correlated.
Alternative term used in the survey sampling literature is Permanent
Random Numbers (PRN).  Coordinated sampling is also a form of locality
sensitive hashing (LSH):  When the weights  in two instances (rows) are very
similar, the samples we obtain are similar, and more likely to be
identical.

The method of coordinating samples had been rediscovered many times,
for different application, in both statistics and computer science. 
The main reason for its consideration by computer
scientists is that it allows for more accurate estimates of queries
that span multiple instances such as distinct counts and similarity
measures
\cite{Broder:CPM00,BRODER:sequences97,ECohen6f,CoWaSu,MS:PODC06,GT:spaa2001,Gibbons:vldb2001,BCMS:ton2004,DDGR07,BHRSG:sigmod07,HYKS:VLDB2008,LiChurchHastie:NIPS2008,CK:sigmetrics09,multiw:VLDB2009}.
In some cases, such as all-distances sketches
\cite{ECohen6f,CoKa:jcss07,MS:PODC06,bottomk07:ds,bottomk:VLDB2008,ECohenADS:PODS2014}
of neighborhoods of nodes in a graph, coordinated samples are obtained
much more efficiently than independent samples.  Coordination can be
efficiently achieved by using a random hash function, applied to the
item key, to generate the seed, in conjunction with the
single-instance scheme of our choice (PPS or Reservoir).  The use of
hashing allows the sampling of different instances to be performed
independently when storing very little state.

 The result of coordinated sampling of different instances when
 restricted to a 
single item is a monotone sampling scheme that is applied to the tuple
$\vecv$ of the weights of the item on the different instances (a
column in our matrix).  
\footnote{Bottom-$k$ samples  
select exactly $k$ items in each instance, hence inclusions of items  
are dependent.
We obtain a single-item restriction by
  considering the sampling scheme for the item conditioned 
on fixing the seed values of other items. 
 A similar situation is with all-distances sketches, where we
can use the HIP inclusion probabilities \cite{ECohenADS:PODS2014}, which
are conditioned on fixing the randomization of all closer nodes.}
The estimation problem of an item-function is a monotone estimation
problem for this sampling scheme.
 
The data domain is a subset of $r\geq 1$ dimensional vectors ${\bf V}
\subset \mathbb{R}_{\geq 0}^r$
(where $r$ is the number of instances in the query specification).
 The sampling is specified by $r$ continuous non-decreasing functions on $(0,1]$:
$\boldsymbol{\tau}=\tau_1,\ldots,\tau_r$.  
The sample $S$ includes the $i$th entry of $\vecv$ with its value $v_i$ if and only 
$v_i \geq \tau_i(u)$.    
Note that when entry $i$ is not sampled, we also 
have some information, as we know that $v_i < \tau_i(u)$. 
Therefore the set $S^*$ of data vectors consistent with our sample
(which we do not explicitly compute) includes 
the exact values of some entries and upper bounds on other entries.
Since the functions $\tau_i$ are non-decreasing, the sampling scheme
is monotone.
In particular, PPS sampling of different instances, restricted to a
single item, is expressed with
$\tau_i(u)$ that are linear functions:  There is a fixed vector
$\boldsymbol{\tau}^*$ such that $\tau_i(u) \equiv u \tau^*_i$.

Coordinated PPS sampling of the instances in Example \ref{example1} is
demonstrated in Example \ref{example2}.   
 The term {\em coordinated} refers to the use of the same random seed $u$ to 
 determine the sampling of all entries in the tuple.  This is in 
 contrast to {\em independent} where a different (independent) seed is 
 used for each entry \cite{CK:pods11}. 



We now return to the original setup of estimating sum aggregates,
such as $L^p_p$.
 Sum aggregates over a domain of items $\sum_{i\in D} f(\vecv^{(i)})$ are  estimated by
summing up estimators for the item function over the selected items,
that is $\sum_{i\in D} \hat{f}(S(\vecv^{(i)},u^{(i)})$.  In general,
the ``sampling'' is very sparse, and we expect that $\hat{f}=0$ for
most items.
These item estimates typically have high 
variance, since most or all of the entries are missing from the 
sample.  We therefore insist on unbiasedness and pairwise 
independence of the single-item estimates.  That way,
$\var[\sum_{i\in D} \hat{f}(S(\vecv^{(i)},u^{(i)}))]= \sum_{i\in D}
\var[\hat{f}(S(\vecv^{(i)},u^{(i)}))]$,
 the variance of 
the sum estimate is the sum over items in $i\in D$ of the variance of 
$\hat{f}$ for $\vecv^{(i)}$. 
Thus (assuming variance is
balanced)  we can expect the relative error to decrease $\propto 1/\sqrt{|D|}$.
Lastly, since the 
functions we are interested in are nonnegative, we also require  the 
estimates to be nonnegative (results 
extend to any one-sided range restriction on the estimates).
Therefore, the estimation of the sum-aggregate is reduced to monotone estimation on
single items.


In  \cite{CKsharedseed:2012} we provided a complete characterization of
estimation problems over coordinated samples for which estimators with desirable global 
properties exist.  This characterization can be extended to  monotone estimation.
 The properties considered were unbiasedness and nonnegativity, and 
together with finite variances or boundedness.  We also showed that 
for any coordinated estimation problem for which an unbiased nonnegative estimator with 
finite variances exists, we can construct an estimator, which we named
the J estimator, that is 
84-competitive.   The J estimator, however, is generally not 
admissible, and also, the construction was geared to establish  $O(1)$
competitiveness rather than obtain a ``natural'' estimator or to minimize 
the constant. 
\subsection*{Contributions}
\ignore{
We aim for developing highly applicable tools,
namely, ``algorithms'' for computing good estimators and for a good 
understanding of the foundations: 
Understand the limits of variance 
competitiveness.  What is the best ratio we can always guarantee ?
Can we achieve this efficiently (and together with admissibility) ?
Secondly, we aim for a better understanding of the ``Pareto front'' of 
admissible estimators which 
satisfy the global properties.  We then hope 
to leverage the freedom we have in estimator selection to {\em customize} the derivation 
and obtain estimators that perform better on recurring patterns we 
can learn or observe in the data. 
}

 The main contributions we make in this paper are the derivation of estimators for general
  monotone estimation problems.  Our estimators are admissible, easy
  to apply, and satisfy desirable properties.  
We now state the main contributions in detail.  We provide pointers to
examples and to the appropriate sections in the body of the paper.
\ignore{
To facilitate a study of competitiveness, in \cite{CKsharedseed:2012} we expressed the $\vecv$-optimal
estimates for nonnegative unbiased estimators.   
We also defined, for any tuple function,  the J
estimator which is unbiased, nonnegative, and has bounded variances, 
provided that 
an estimator with a subset of these properties exists for the function.
The J estimator is also variance competitive, but 
not variance$^+$ optimal, where we define variance$^+$ optimality as
variance optimality over unbiased nonnegative estimators.
}



\medskip
\noindent
{\bf The optimal range:} 
We start by defining the admissibility playing field for
 unbiased nonnegative estimators.   
We define the {\em optimal range} of estimates
(Section~\ref{poptr:sec}) for each particular outcome,   {\em conditioned}
on the estimate values on all ``less-informative'' outcomes (outcomes
which correspond to larger seed value $u$).   
The range includes all
estimate values that are ``locally'' optimal with respect to at least one data vector
that is consistent with the outcome.
We show that being ``in range'' almost everywhere is 
necessary for admissibility and is sufficient for unbiasedness
and nonnegativity, when an unbiased nonnegative estimator exists.

\medskip
\noindent
{\bf The \L\ estimator:}
  The lower extreme of the optimal range is obtained
by solving the constraints that force 
the estimate on each outcome to be equal to the infimum of the optimal range.
We refer to this solution as the
{\em \L\ estimator}, and study it extensively in Section~\ref{Lest:sec}.  

We show that the \L\ estimator, which is the 
solution of a respective integral equation, can be expressed in the
following convenient form:
\begin{align}
\hat{f}^{(L)}(S,\rho) &= \frac{\underline{f}^{(\vecv)}(\rho)}{\rho}- \int_{\rho}^1
\frac{\underline{f}^{(\vecv)}(u)}{u^2} du\ , \\
\end{align}
where $\rho$ is the seed value used to obtain the sample $S$, 
$\vecv\in S^*$ is any (arbitrary)
data vector consistent with $S$ and $\rho$, and the {\em lower bound} function 
$\underline{f}^{(\vecv)}(u)$ is defined as the infimum of $f(\vecz)$ over all vectors
$\vecz\in S^*(\vecv,u)$ that are consistent with the sample obtained for data $\vecv$ with seed $u$.
We note that the estimate is the same for any choice of $\vecv$ and that
the values $\underline{f}^{(\vecv)}(u)$ for all $u\geq \rho$ 
can be computed from $S$ and $\rho$.  Therefore, the estimate is well defined.
  This expression allows us to
efficiently compute the estimate, for any function,  by numeric
integration or a closed form (when a respective definite integral has a closed form).
The lower bound function is presented more precisely in Section \ref{prelim:sec}
and an example is provided in Example \ref{example3}.  An example derivation of the \L\ estimator for the functions $\range_{p+}$ is provided in Example \ref{example4}.

We show that the \L\
estimator has a natural and compelling
combination of properties.  It
satisfies both our quality measures, being both admissible
and $4$-competitive for any instance of the monotone estimation problem
for which a bounded variance estimator exists.  The competitive ratio
of $4$  improves over the previous upper bound of 84
\cite{CKsharedseed:2012}.   
We show that the ratio of $4$ of the \L\ estimator is tight in the sense that there is a family of
functions on which the supremum of the ratio, over functions and data
vectors,  is $4$.   We note however that
the \L\ estimator has lower ratio  for
specific functions.  For example, we computed 
ratios of 2 and 2.5, respectively, 
for exponentiated range with $p=1,2$ (Which
facilitates estimation of $L_p$ differences, see Example \ref{example1}).

Moreover, the \L\ estimator is
{\em monotone}, meaning that when fixing the data
vector, the estimate value is monotone non-decreasing with the
information in the outcome (the set $S^*$ of data vectors that are consistent
with our sample).  In terms of our monotone sampling formulation, 
estimator monotonicity means
that when we fix the data $\vecv$, the estimate is non-increasing with
the seed $u$.   Furthermore,  the \L\ estimator is the
{\em unique} admissible monotone estimator and thus dominates
(has at most the variance on every data vector) the Horvitz-Thompson (HT) estimator~\cite{HT52} (which is also unbiased,
nonnegative, and monotone).  

 To further illustrate this point, recall that the HT estimate is positive only on outcomes when we know
  $f(\vecv)$.  In this case, we have the inverse probability
  estimate $f(\vecv)/p$, where $p$ is the probability of an outcome
which reveals  $f(\vecv)$.  When we have partial information on
$f(\vecv)$, the HT estimate does not utilize that and is $0$ whereas
admissible estimators, such as the \L\ estimators, must use this
information.
It is also possible that the probability of an outcome that reveals
$f(\vecv)$ is $0$.  In this case, the HT estimator is not even applicable.
  One natural example is estimating the range $|v_1-v_2|$ with
say $\tau_1(u)\equiv \tau_2(u) \equiv u$, this is essentially classic
Probability Proportional to Size (PPS) sampling (coordinated between ``instances'').  When the input is
$(0.5,0)$, the range is $0.5$, but there is $0$ probability of
revealing $v_2=0$.   We can obtain nontrivial lower (and upper) bounds on
the range:  When $u\in (0,0.5)$, we have a lower bound of $0.5-u$.
Nonetheless, the probability of knowing the exact value ($u=0$) is $0$.
In contrast to the HT estimate,  our \L\ estimator is defined
for any monotone estimation instance for which a nonnegative unbiased
estimator with finite variance exists.

\medskip
\noindent
{\bf Order-optimal estimators:}
In many situations we have information on data patterns.  For example, if
our data consists of  hourly temperature measurements across locations or
daily summaries of Wikipedia, we expect it to be fairly stable.  That is,
we expect instances  to
be very similar.
That is, most tuples of values , each corresponding to a particular
geographic location or Wikipedia article, would have most entries being very similar.
In other cases, such as IP traffic, differences are typically larger.
Since there is a choice, the full Pareto front of admissible estimators, 
we would like to be able to select an estimator that would have lower
variance on more likely patterns of data vectors, 
this while still providing some weaker ``worst case'' guarantees for all applicable data vectors in our domain.

 Customization of estimators to data patterns can be facilitated through  {\em order
  optimality}~\cite{CK:pods11}.  More precisely, 
an estimator is $\prec^+$-optimal with respect to some partial order $\prec$  on data vectors
if any other (nonnegative unbiased) estimator with lower variance
on some data $\vecv$ must have strictly higher variance on some
data that precedes $\vecv$.  Order-optimality implies
admissibility, but not vice versa.  Order-optimality also uniquely specifies an
admissible estimator.  By specifying an order which
prioritizes more likely patterns in the data, we can customize the
estimator to these patterns.

 We show (Section \ref{estPREC:sec}) how to construct a $\prec^+$-optimal nonnegative unbiased
 estimators for {\em any} function and order $\prec$ for which such estimator
 exists. We show that when the data domain is discrete,
such estimators always exist whereas continuous domains
require some natural convergence properties of $\prec$.

  We also show that the \L\ estimator is $\prec^+$-optimal with respect to the
  order $\prec$ such that $\vecz\prec
  \vecv \iff f(\vecz) < f(\vecv)$.    This means that when estimating the exponentiated
range function, the \L\ estimator is optimized for high similarity (this while providing a strong 4-competitiveness guarantee even for highly dissimilar data).

\medskip
\noindent
{\bf The \U\ estimator:}
We also explore the upper extreme of the optimal range, that is, the solution obtained by aiming for the supremum of the range.  We call this solution the {\em \U\
  estimator} and we study it in Section~\ref{Uest:sec}.  This estimator is
unbiased, nonnegative, and has finite variances.  We formulate some
conditions  on the tuple function, that are satisfied by natural
functions including the exponentiated range, under which the estimator
is admissible.
The \U\ estimator, under some
 conditions, is $\prec^+$-optimal with respect to the order $\vecz\prec
  \vecv \iff f(\vecz) > f(\vecv)$.  In the context of the exponentiated range,
it means that it is optimized for highly dissimilar instances.


\ignore{
Moreover,  the study also demonstrates the significance of
 being able to understand the worst case performance and to tailor estimators to patterns in data through
 $\prec^+$-optimality, competitiveness, and variance optimality.
The tradeoff
 between the \L\ and \U\ estimators difference estimators shows that the \U\
 estimator outperforms the \L\ estimator in applications when ranges are
 large (data sets are less similar) and vice versa.  The experiments
 also demonstrate the value of the competitiveness of the \L\
 estimators:  even on data when they are dominated by tailored
 estimators, they are not ``too far off'' the optimum whereas non-competitive
 estimators can potentially perform much worse when the data deviates
 from the template we expect.
This suggests a default use of the \L\ estimator when there is no
 prior knowledge on the typical form of the data.  The competitive
 ratio of the \L\ estimator is $2$ for $p=1$ and $2.5$ for $p=2$.  
}

 Lastly, in Section \ref{conclu:sec} we conclude with a discussion of
 future work and of 
follow-up uses of our estimators in applications, including
 pointers to experiments.  One application of particular  importance that
 is enabled by our work here is the estimation of $L_p$ difference norms
 over sampled data.  Another application is similarity estimation in social
networks. We hope and believe that our methods and estimators, once understood,
will be more extensively applied.

\begin{fexample*}
{\em Instances $i\in\{1,2,3\}$ and items $k\in\{a,b,c,d,e,f,g,h\}$:}

{\small 
\begin{tabular}{c|llllllll}
       &     a &    b &     c &     d &     e &     f&     g &     h 
       \\
\hline
$v_1$   & $0.95$ & $0$ & $0.23$ & $0.70$ & $0.10$ & $0.42$ & $0$ & $0.32$
   \\
$v_2$   & $0.15$ & $0.44$ & $0$ & $0.80$ & $0.05$ & $0.50$ & $0.20$ & $0$
\\
$v_3$   & $0.25$ & $0$ & $0$ & $0.10$ & $0$ & $0.22$ & $0$ & $0$ 
\end{tabular}
}
\smallskip

{\em Example queries} over selected items
$H\subset[a\text{-}h]$. $L_p$ difference,  $L_p^p$, which is the $p$th power of $L_p$
difference and a sum aggregate which can be used to estimate the $L_p$ difference,
$L^p_{p+}$: asymmetric (increase only) $L^p_{p}$, the sum of the
increase-only and the decrease-only changes  
(decrease only is obtained by switching the roles of $v_1$ and $v_2$) is $L_p^p$,
 but each component is a useful metric for asymmetric change. 
$G$ an ``arbitrary''  sum aggregate, illustrating versatility of queries.
{\small
\begin{align*}
L_p(H) & =(\sum_{k\in H} |v^{(k)}_1-v^{(k)}_2|^p)^{1/p} \\
L_p^p(H) &= \sum_{k\in H} |v^{(k)}_1-v^{(k)}_2|^p \\
L^p_{p+}(H) &= \sum_{k\in H} \max\{0,v^{(k)}_1-v^{(k)}_2\}^p\\
G(H) &= \sum_{k\in H} |v^{(k)}_1-2v^{(k)}_2+ v^{(k)}_3|^2
\end{align*}

\begin{tabular}{l|l}
{\em sum aggregate} & {\em  item function} \\
\hline\hline
$L_p^p$  &    $\range_p(\vecv) = (\max(\vecv)-\min(\vecv))^p$\\
\hline
$L^p_{p+}$    &$\range_{p+}(v_1,v_2) = \max\{0,v_1-v_2\}^p$\\
\hline
$G$  &     $g(v_1,v_2,v_3) = |v_1+v_3-2v_2|^2$
\end{tabular}

\begin{align*}
L_1(\{b,c,e\}) =& |0-0.44|+|0.23-0| + |0.10-0.05| = 0.71 \\
L^2_2(\{c,f,h\}) =& (0.23-0)^2 + (0.50-0.42)^2+(0.32-0)^2 \approx 0.16
\\
L_2(\{c,f,h\}) =& \sqrt{L^2_2 (\{c,f,h\})} \approx 0.40 \\
L_{1+}(\{b,c,e\}) =& \max\{0,0-0.44\} +
\max\{0,0.23-0\}+\\ & +\max\{0,0.10-0.05\}=0.235\\
G(\{b,d\}) =& |0-2*0.44+0|^2+ |0.7-2*0.8+0.1|^2 \approx 1.18
\end{align*}

}
\caption{Dataset with 3 instances and queries  \label{example1}}
\end{fexample*}

\begin{fexample*}
Consider shared-seed coordinated sampling, where each of the instances
1,2,3 is  PPS sampled with threshold $\tau^*=1$.
In this particular case, each entry is sampled with probability equal
to its value.
To coordinate the samples, we draw $u^{(k)} \in U[0,1]$, 
independently for different items.
An item $k$ is sampled in instance $i$ if and only if 
$v^{(k)}_i \geq u^{(k)}$.
$S^{*(k)}$ contains all vectors consistent with the sampled entries
and with value at most $u^{(k)}$ in unsampled entries.

{\small
\begin{tabular}{c|llllllll}
item       &     a &    b &     c &     d &     e &     f&     g &     h
       \\
\hline
$v_1$   & $0.95$ & $0$ & $0.23$ & $0.70$ & $0.10$ & $0.42$ & $0$ & $0.32$
   \\
$v_2$   & $0.15$ & $0.44$ & $0$ & $0.80$ & $0.05$ & $0.50$ & $0.20$ & $0$
\\
$v_3$   & $0.25$ & $0$ & $0$ & $0.10$ & $0$ & $0.22$ & $0$ & $0$ \\
\hline
$u^{(k)}$ & $0.32$ & $0.21$ & $0.04$ & $0.23$ & $0.84$ & $0.70$ & $0.15$ & $0.64$ 
\end{tabular}
}

 The outcomes for the different items are:
$S^{(a)}=(0.95,*,*)$, $S^{(b)}=(*,0.44,*)$, 
$S^{(c)}=(0.23,*,*)$, $S^{(d)}=(0.7,0.8,*)$, 
$S^{(e)}=S^{(f)}=S^{(h)}=(*,*,*)$, 
$S^{(g)}=(*,0.2,*)$.
The sets of vectors consistent with the outcomes are
$S^{*(a)}=\{0.95\}\times [0,0.32)^2$  and 
$S^{*(h)}=[0,0.64)^3$.

\caption{Coordinated PPS sampling for Example~\ref{example1}\label{example2}}
\end{fexample*}

\begin{fexample*}
Consider 
$\range_{p+}(v_1,v_2)=\max\{0,v_1-v_2\}^p$ 
(see Example \ref{example1}) over the domain
${\bf V}=[0,1]^2$ and PPS
sampling with $\tau^*_1=\tau^*_2=1$
(as in Example \ref{example2}).
The lower bound function for data $\vecv=(v_1,v_2)$ is
$$\underline{\range_{p+}}(u,\vecv) =\max\{0,v_1-\max\{v_2,u\}\}^p\ .$$
The figures below illustrate $\underline{\range_{p+}}^{(\vecv)}(u)$
(LB) and its
lower hull (CH) for the data vectors
$(0.6,0.2)$ and $(0.6,0)$ and $p=\{0.5,1,2\}$.
For $u> 0.2$, the outcome when sampling both vectors is the same, and
thus the lower bound function is the same.  For $u\leq 0.2$, the
outcomes diverge. 
For $p\leq 1$, $\underline{\range_{p+}}^{(\vecv)}(u)$ is concave and
the lower hull is linear on $(0,v_1]$.  For $p>1$, the lower hull
coincides with $\underline{\range_{p+}}^{(\vecv)}(u)$ on some interval
$(a,v_1]$ and is linear on $(0,a]$.  When $v_2=0$,
$\underline{\range_{p+}}^{(\vecv)}(u)$ is equal to its lower hull.
\centerline{
\begin{tabular}{ccc}
\ifpdf
\includegraphics[width=0.32\textwidth]{phalf2_LB_CH}
\else
 \epsfig{figure=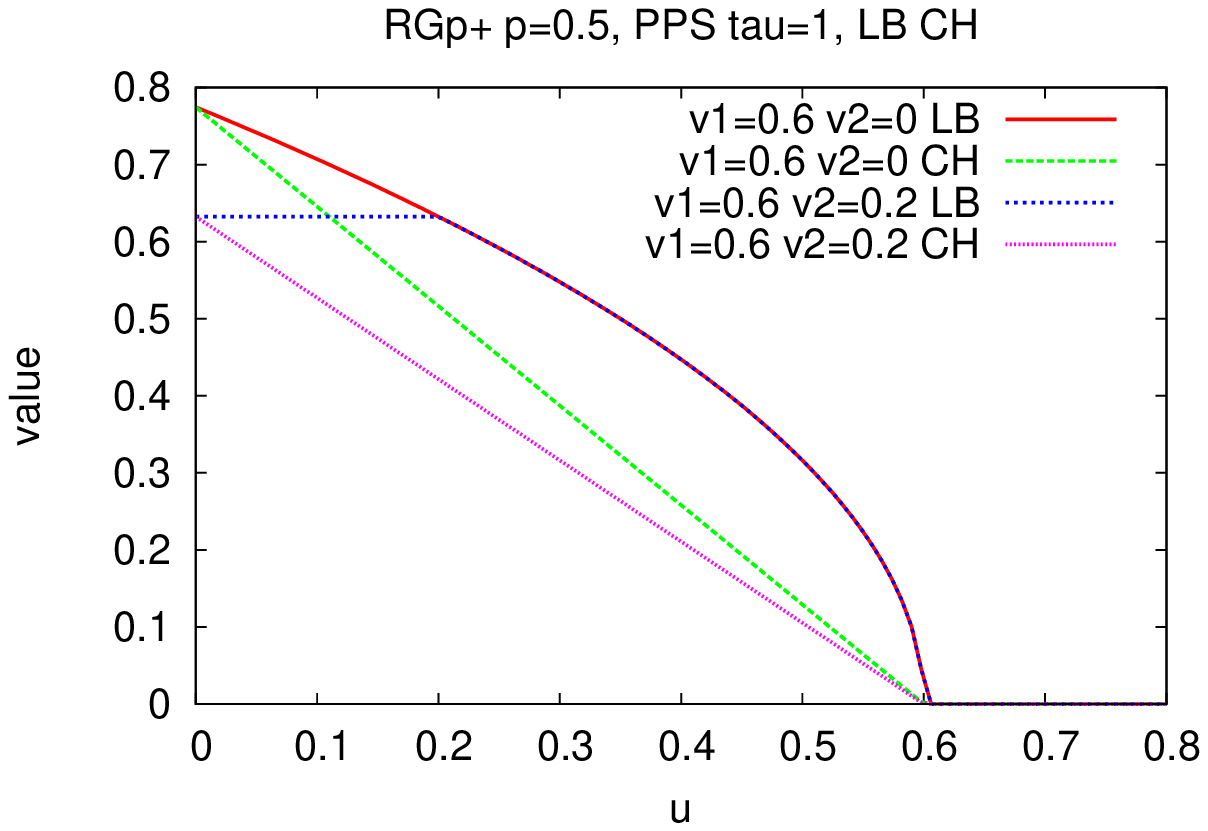,width=0.32\textwidth}
\fi
&
\ifpdf
\includegraphics[width=0.32\textwidth]{pone2_LB_CH}
\else
 \epsfig{figure=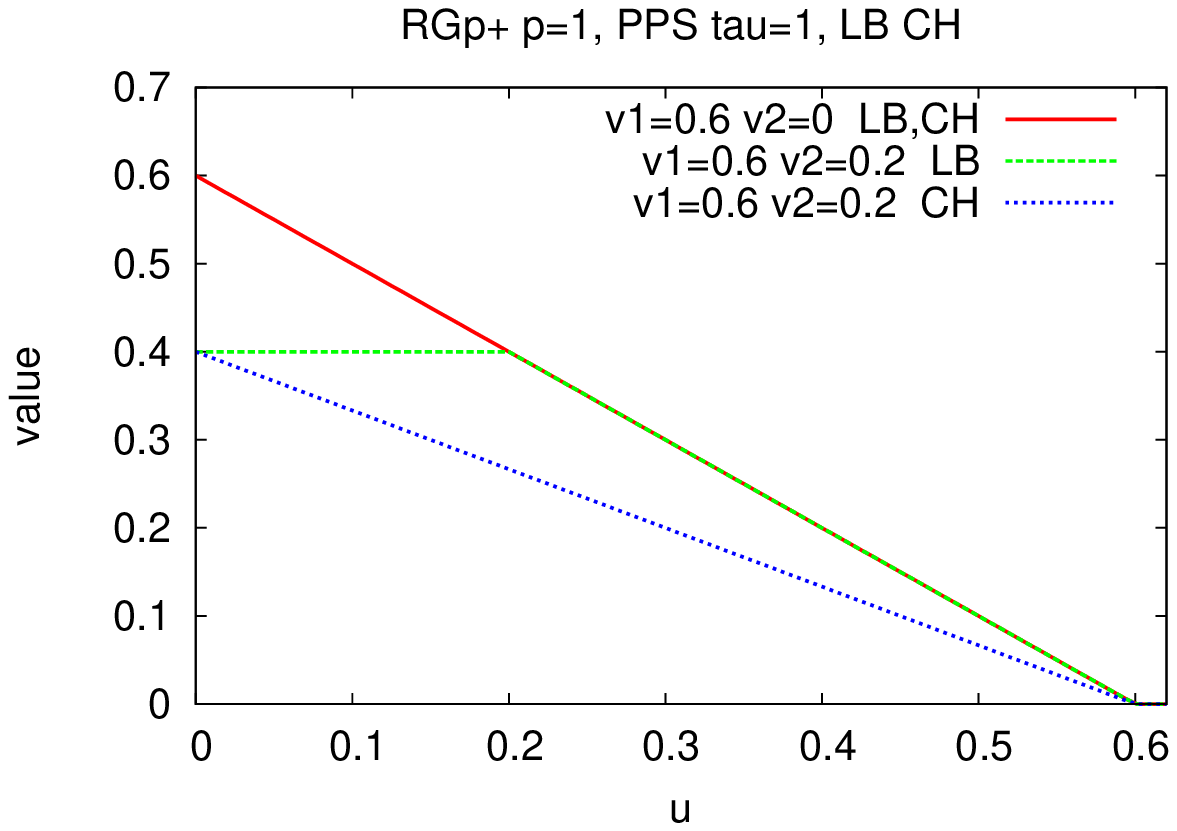,width=0.32\textwidth}
\fi
&
\ifpdf
\includegraphics[width=0.32\textwidth]{ptwo2_LB_CH}
\else
 \epsfig{figure=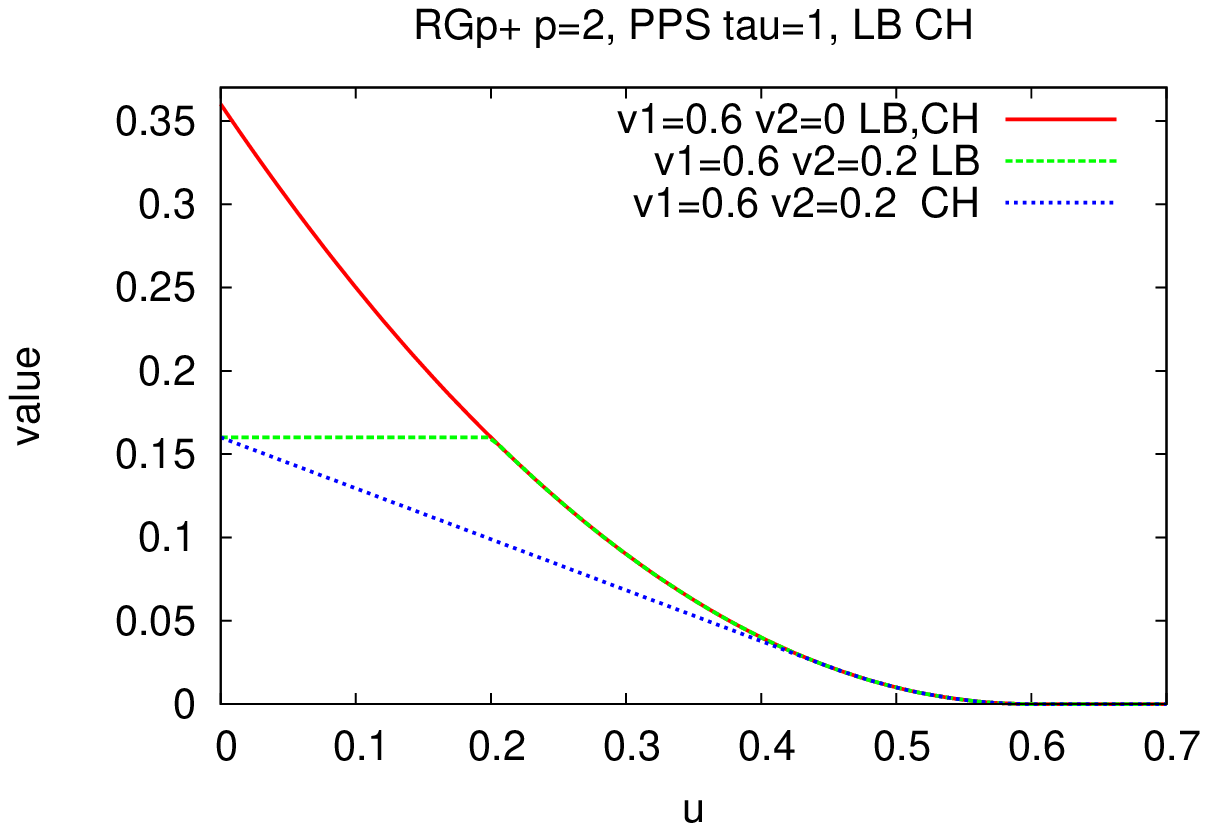,width=0.32\textwidth}
\fi
\end{tabular}
}
The $\vecv$-optimal estimates are the negated slopes of
the lower hulls.  They are $0$ when $u\in (0.6,1]$, since these
outcomes are consistent with data on which $\underline{\range_{p+}}=0$.
They are constant for $u\in (0,v_1]$ when $p\leq 1$. 
Observe that for $u\in (0.2,0.6]$, the $\vecv$-optimal estimates are
different even though the outcome of sampling the two vectors are
the same --  demonstrating that it is not possible to
simultaneously minimize the variance of the two vectors.
\caption{Lower bound function and its lower hull\label{example3}}
\end{fexample*}

\begin{fexample*}
We compute the \L\ and \U\ estimators for $\range_{p+}$ for the
sampling scheme and data in Example \ref{example3}.  
For the two vectors $(0.6,0.2)$ and $(0.6,0)$, both
the \L\ and \U\ estimates are $0$ when $u\geq 0.6$, this
is necessary from unbiasedness and nonnegativity because for these outcomes
$\exists \vecv\in S^*, \range_{p+}(\vecv)=0$.
Otherwise, the \L\ estimate is
$\hat{\range}^{(L)}_{p+}(S)=(v_1-{v'}_2)^p/{v'}_2 - \int_{{v'}_2}^{v_1} \frac{(v_1-x)^p}{x^2}dx$,
where ${v'}_2=u$ when  $S=\{1\}$ and ${v'}_2=v_2$ when $S=\{1,2\}$.
When $p\geq 1$, the \U\ estimate is 
$\hat{\range}^{(U)}_{p+}(S)=p(v_1-u)^{p-1}\ $
when $u\in (v_2,v_1]$ and $0$ when $u\leq v_2<v_1$.
When $p\leq 1$ the \U\ estimate is $v_1^{p-1}$ when $u\in (v_2,v_1]$ 
and  $\frac{(v_1-v_2)^p-v_1^{p-1}(v_1-v_2)}{v_2}$  when $u\leq v_2<v_1$.

 The figure also include the $\vecv$-optimal estimates, discussed in Example \ref{example3}.
\ignore{When $p\leq 1$, the lower bound function
$\underline{\range_{p+}}^{(\vecv)}(u)$
 is concave, and its lower hull is a linear function between
 $(0,(v_1-v_2)^p)$ and $(v_1,0)$.  The $\vecv$-optimal estimates are
$(v_1-v_2)^p/v_1$ for $u\in (0,v_1)$.
}
When $v_2=0$, the \U\ estimates are $\vecv$-optimal.
The \L\ estimate is not bounded when $v_2=0$ (but has bounded variance
and is competitive).

\centerline{
\begin{tabular}{ccc}
\ifpdf
\includegraphics[width=0.32\textwidth]{phalf2_OLU_ests}
\else
 \epsfig{figure=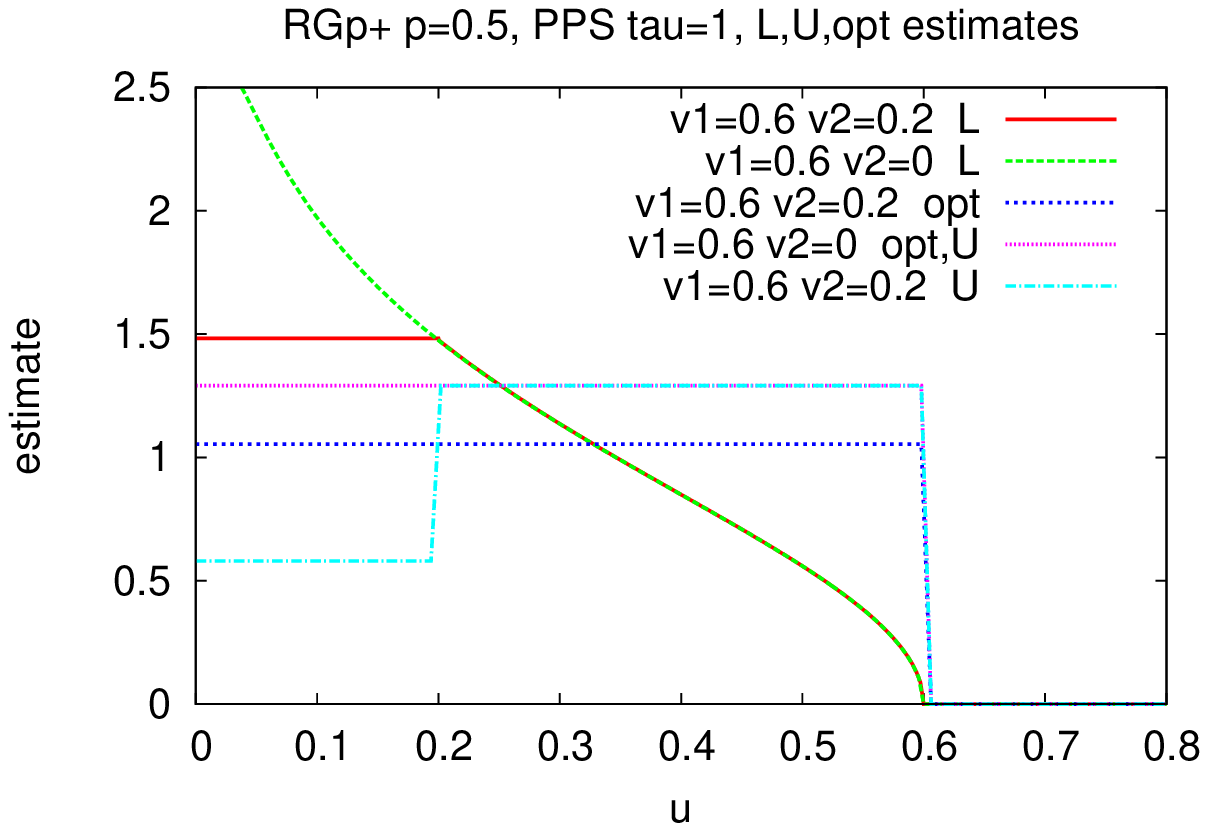,width=0.32\textwidth}
\fi
&
\ifpdf
\includegraphics[width=0.32\textwidth]{pone2_OLU_ests}
\else
 \epsfig{figure=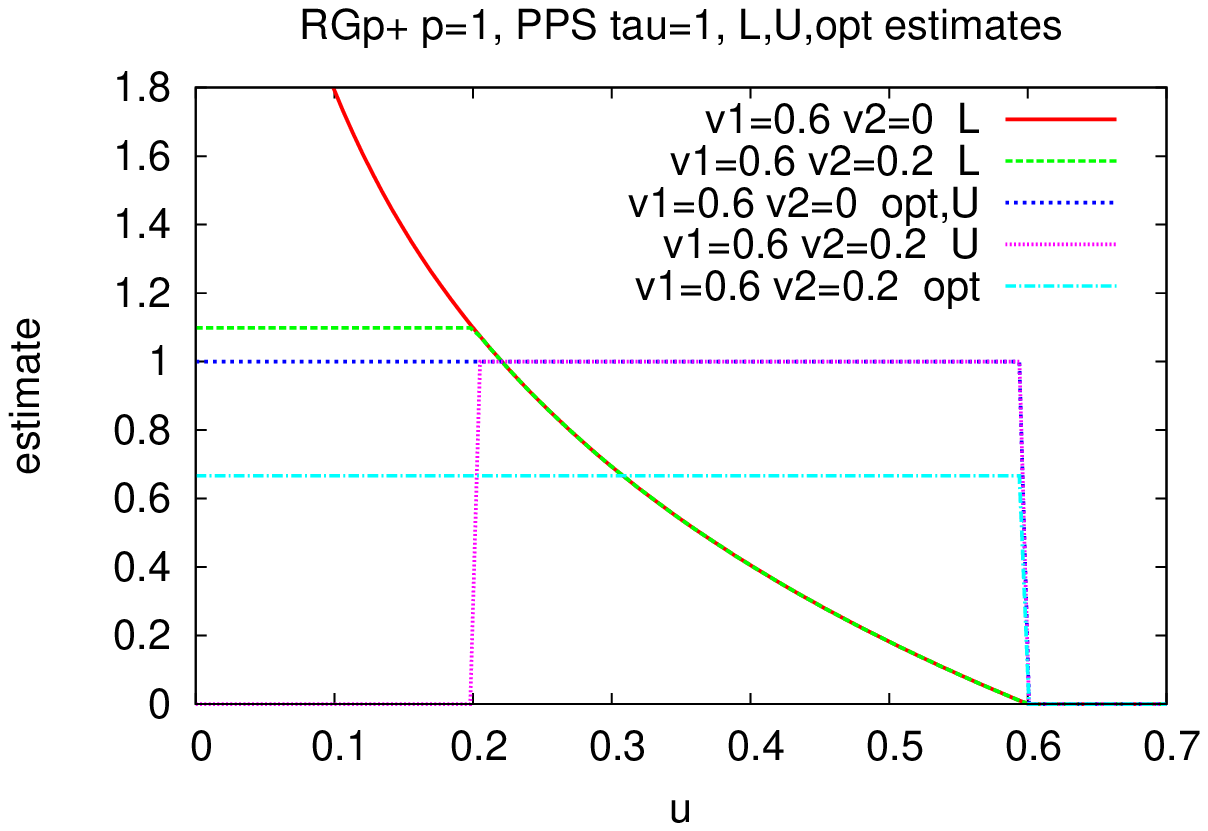,width=0.32\textwidth}
\fi
&
\ifpdf
\includegraphics[width=0.32\textwidth]{ptwo2_OLU_ests}
\else
 \epsfig{figure=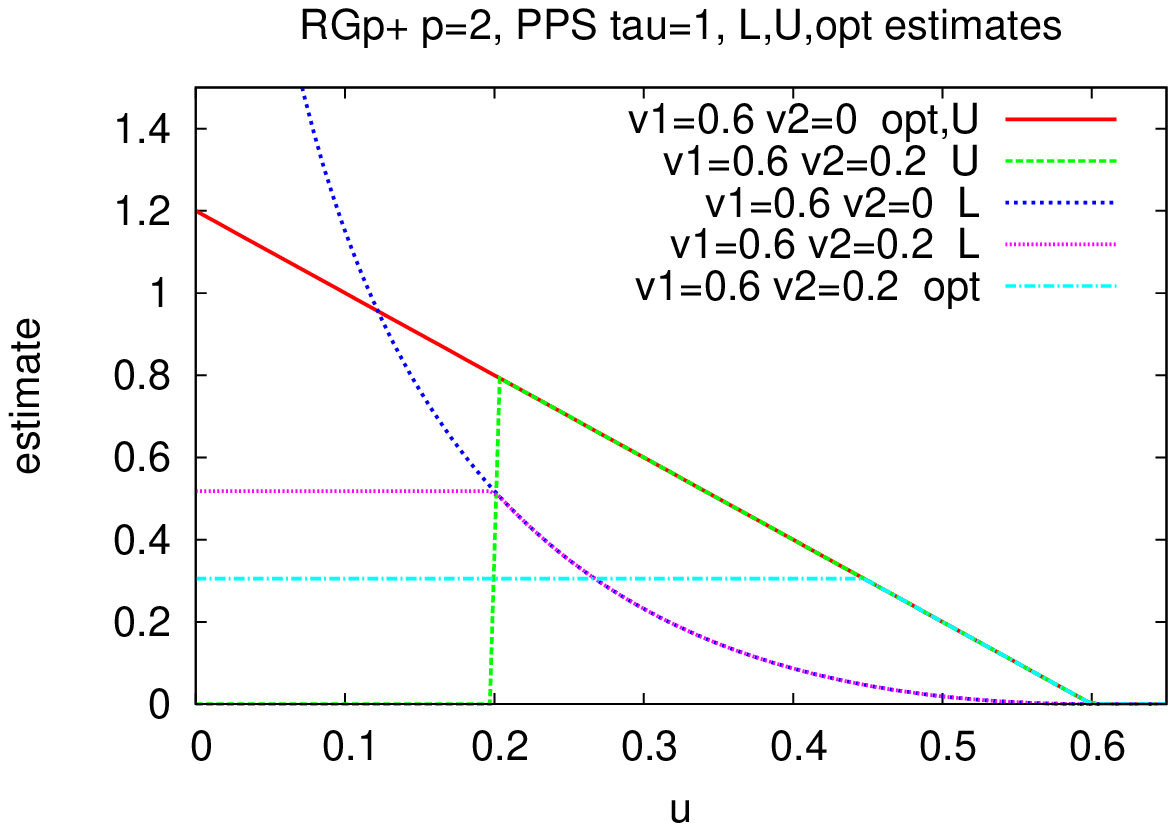,width=0.32\textwidth}
\fi
\end{tabular}
}
\caption{\L\ and \U\ estimates for Example \ref{example3}\label{example4}}
\end{fexample*}

\begin{fexample*}
{\small
We derive $\prec^+$-optimal $\range_{1+}$
estimators over the finite domain ${\bf V}=\{0,1,2,3\}^2$.  
Assuming same sampling scheme on both entries, there are 3 threshold
values of interest, where $\pi_i$ $i\in [3]$ is such that entry of
value $i$ is sampled if and only if $u\leq \pi_i$.  We have
$\pi_1<\pi_2<\pi_3$.

 The lower bounds  $\underline{\range_{1+}}^{(\vecv)}$ are 
step functions with
steps at $u=\pi_i$.  The table below shows
$\underline{\range_{1+}}^{(\vecv)}(u)$ for all $u$ and $\vecv$ such that
$\range_{1+}(\vecv)>0$.  When $\range_{1+}(\vecv)=0$, we have 
$\underline{\range_{1+}}^{(\vecv)}(u)\equiv 0$ and any unbiased nonnegative
estimator must have $0$ estimates on outcomes that are consistent with $\vecv$.

\begin{tabular}{l|cccccc}
$ \underline{\range_{1+}}^{(\vecv)}$   & $(1,0)$   &  $(2,1)$ &  $(2,0)$ & $(3,2)$ & $(3,1)$ & $(3,0)$ \\
\hline
$(0,\pi_1]$ & $1$ & $1$  & $2$  &  $1$  & $2$ & $3$ \\
$(\pi_1,\pi_2]$ & $0$ & $1$  & $1$  &  $1$  & $2$ & $2$ \\
$(\pi_2,\pi_3]$ & $0$ & $0$  & $0$  &  $1$  & $1$ & $1$ \\
$(\pi_3,1]$ & $0$ & $0$  & $0$  &  $0$  & $0$ & $0$ 
\end{tabular}

\medskip
The
$\vecv$-optimal estimate, $\hat{\range}_{1+}^{(\vecv)}(u)$ is the
negated slope at $u$ of the lower hull of
$\underline{\range_{1+}}^{(\vecv)}$. 
The lower hull of each
step function is piecewise linear with breakpoints at a
subset of $\pi_i$, and thus, the $\vecv$-optimal estimates are
constant on each segment $(\pi_{i-1},\pi_i]$.
 The table
shows the estimates for all $\vecv$ and $u$.  
The notation $\downarrow$ refers to value in same column and one row
below and $\Downarrow$ to value two rows below.

\begin{tabular}{l|cccccc}
 $\hat{\range}_{1+}^{(\vecv)}$   & $(1,0)$   &  $(2,1)$ &  $(2,0)$ & $(3,2)$ & $(3,1)$ & $(3,0)$ \\
\hline
$(0,\pi_1]$ & $\frac{1}{\pi_1}$ & $\frac{1}{\pi_2}$  &  $\frac{2-(\pi_2-\pi_1)\downarrow}{\pi_1}$   &  $\frac{1}{\pi_3}$  & $\frac{2-\Downarrow}{\pi_2}$ & $\frac{3-\downarrow(\pi_3-\pi_2)-\Downarrow(\pi_2-\pi_1)}{\pi_1} $ \\
$(\pi_1,\pi_2]$ & $0$ &  $\frac{1}{\pi_2}$  &  $\min\{\frac{2}{\pi_2},\frac{1}{\pi_2-\pi_1}\}$  &  $\frac{1}{\pi_3}$  & $\frac{2-\downarrow}{\pi_2} $ & $\min\{\frac{3-\downarrow(\pi_3-\pi_2)}{\pi_2},\frac{2-\downarrow(\pi_3-\pi_2)}{\pi_2-\pi_1} $ \\
$(\pi_2,\pi_3]$ & $0$ & $0$  & $0$  &  $\frac{1}{\pi_3}$  & $\min\{\frac{2}{\pi_3},\frac{1}{\pi_3-\pi_2}\} $ & $\min\{\frac{3}{\pi_3},\frac{1}{\pi_3-\pi_2}\} $ 
\end{tabular}

\medskip

  The order 
$(2,1)\prec (2,0)$ and 
$(3,2)\prec (3,1) \prec (3,0)$ yields the \L\ estimator, which is
$\vecv$-optimal for $(1,0)$, $(2,1)$, and $(3,2)$.
The order
$(2,0)\prec (2,1)$ and 
$(3,0)\prec (3,1) \prec (3,2)$ yields the \U\ estimator
which is
$\vecv$-optimal for $(1,0)$, $(2,0)$, and $(3,0)$.
Observe that it suffices to only specify $\prec$ so that 
the order is defined between vectors consistent with the same outcome $S$
when $\underline{\range_{1+}}(S)>0$.
For $\range_{1+}$, this means specifying 
the order between vectors with the same $v_1$ value
(and only consider those with strictly smaller $v_2$).  In follows that
any admissible estimator is $(1,0)$-optimal.

 To specify an estimator, we need to specify it on all possible
 outcomes, where each distinct outcome is 
uniquely determined by a corresponding set of data vectors  $S^*$.  The 8 possible outcomes
(we exclude those consistent with vectors with $\range_{1+}(\vecv)=0$ on which the estimate must
be $0$) are $(1,0)$, $(2,\leq 1)$, $(2,1)$, $(3,\leq 2)$, $(3,2)$,
$(3, \leq 1)$, $(3,1)$, and $(3,0)$, where an entry ``$\leq a$''
specifies all vectors in ${\bf V}$ where the entry is at most $a$.

We show how we construct the $\prec^+$-optimal estimator for $\prec$
which prioritizes vectors with difference of $2$:
$(3,1)\prec (3,2) \prec (3,0)$ and $(2,0)\prec (2,1)$.
The estimator is $\vecv$-optimal for $(3,1)$, $(2,0)$, and $(1,0)$.
This determines the estimates 
$\hat{\range}_{1+}^{(\prec)}$ 
on all outcomes consistent with these vectors: The value 
on outcome $(1,0)$ is  $\hat{\range}^{((1,0))}((0,\pi_1])$,  the
values on outcomes $(2,\leq 1)$ and $(2,0)$ are according to
$\hat{\range}^{(2,0)}$ on $(\pi_1,\pi_2]$ and $(0,\pi_1]$,
respectively,
 and value on the outcomes
$(3,\leq 2)$, $(3, \leq 1)$ and $(3,1)$  is according to $\hat{\range}^{(3,1)}$
on $(\pi_2,\pi_3]$ and $(\pi_1,\pi_2]$.    These values are provided
in the table above.
The remaining outcomes are
$(3,0)$, $(3,2)$, and $(2,1)$.   We need to specify the estimator so
that it is unbiased on these vectors, given the existing specification.
We have
\begin{align*}
\hat{\range}_{1+}^{(\prec)}(2,1) & =\frac{1-(\pi_2-\pi_1)\hat{\range}_{1+}^{(\prec)}(2,\leq 1)}{\pi_1}\\
\hat{\range}_{1+}^{(\prec)}(3,0) &=\frac{3-(\pi_3-\pi_2)\hat{\range}_{1+}^{(\prec)}(3,\leq 2)-(\pi_2-\pi_1)\hat{\range}_{1+}^{(\prec)}(3,\leq 1) }{\pi_1}\\
\hat{\range}_{1+}^{(\prec)}(3,2)
&=\frac{2-(\pi_3-\pi_2)\hat{\range}_{1+}^{(\prec)}(3,\leq
  2) }{\pi_1}\ .
\end{align*}

Observe that to apply these estimators, we do not have to precompute the estimator on all
possible outcomes.   An estimate only depends on values of the
estimate on all less informative outcomes.  In a discrete domain
as in this example, it is the number of breakpoints larger than 
the seed $u$ (which is at most the number of distinct values in the domain).  
}
\caption{Walk-through derivation of $\prec^+$-optimal estimators\label{precex}}
\end{fexample*}

\section{Preliminaries}  \label{prelim:sec}
  We present some properties of monotone sampling and briefly review concepts and of results from
  \cite{CK:pods11,CKsharedseed:2012} which we build upon here.

\smallskip
\noindent
\ignore{
{\bf Sampling model:}
The data domain ${\bf V } \subset \mathbb{R}^r$ is a subset of the
reals bounded from below.
The 
{\em sampling scheme}  is specified by
continuous non-decreasing functions
$\mbox{\boldmath{$\tau$}}=(\tau_1,\ldots,\tau_r)$ on $[0,1]$
where the infimum of the range of $\tau_i$ is at most $\inf_{\vecv\in {\bf V}}  v_i$.
To apply the sampling to a data vector $\vecv=(v_1,v_2,\ldots,v_r)\in {\bf V}$,
 we draw 
a uniform random number
$u\sim U[0,1]$, which we refer to as the {\em seed}.
The output of the sampling
is the outcome $S\equiv S(u,\vecv)$, which is a subset of the entries
of $\vecv$ such that
 the $i$th entry of $\vecv$ is included in $S$ if and only if $v_i$ is at least $\tau_i(u)$:
$$i\in S\,  \iff \, v_i\geq \tau_i(u) \ .$$ We also assume that the
seed $u$ is available with the outcome.

  With each outcome $S(u,\vecv)$, we can identify the set 
$S^*$  of all data vectors that are consistent with it:
\begin{align*}
& S^*\equiv V^*(u,\vecv) =  \\ & 
\{\vecz \mid \forall i\in [r], i\in S\wedge   z_i=v_i
 \, \vee \, i\not\in S \wedge z_i< \tau_i(u)   \}\ .
\end{align*}
}

Consider monotone sampling, as defined in the introduction.
 For any two 
outcomes, $S^*_1=S^*(u,\vecv)$ and $S^*_2=S^*(u',\vecv')$, 
the sets $S^*_1$ and $S^*_2$  must be either disjoint or one is
contained in the other.   This is because if there is a common data
vector $\vecz \in S^*_1\cap S^*_2$, then $S^*_1=S^*(u,\vecz)$ and
$S^*_2=S^*(u',\vecz)$.
From definition of monotone sampling, if $u'>u$ then $S^*_1\subseteq
S^*_2$ and vice versa.
For any $\vecv, \vecz\in {\bf V}$ , the set of $u$ values which satisfy
$S^*(u,\vecv)=S^*(u,\vecz)$ is a suffix of the interval $(0,1]$.
This is because $S^*(u,\vecv)=S^*(u,\vecz)$ implies
$S^*(u',\vecv)=S^*(u',\vecz)$ for all $u'>u$.
For convenience, we assume that this interval is open to the left:
\footnote{This assumption can be integrated while affecting at most a ``zero 
measure'' set of outcomes for any data point.  Therefore, this does not affect estimator properties.}
\begin{eqnarray}
 \lefteqn{\forall \rho\in (0,1]\ \forall \vecv ,}
\label{openset:lemma}\\
& \vecz \in S^*(\rho,\vecv) \implies
\exists \epsilon > 0,\ \forall x\in (\rho-\epsilon,1],\ \vecz\in
S^*(x,\vecv)\nonumber
\end{eqnarray}


\smallskip
\noindent
{\bf Estimators:}
 We are interested in estimating, from the outcome $S(u,\vecv)$, the
 quantity $f(\vecv)$, where the function
$f:{\bf V}$ maps ${\bf V}$ to the nonnegative reals.
We apply an {\em estimator} $\hat{f}$ to the outcome (including the seed)
and use the notation $\hat{f}(u,\vecv)\equiv
\hat{f}(S(u,\vecv))$.  When the domain is continuous, we assume
$\hat{f}$ is (Lebesgue) integrable.

Two estimators $\hat{f}_1$ and $\hat{f}_2$
are {\em equivalent} if for all data $\vecv$, $\hat{f}_1(u,\vecv)=\hat{f}_2(u,\vecv)$ with probability $1$, which is the same as
\begin{align} 
\text{$\hat{f}_1$ and $\hat{f}_2$
are equivalent} \iff
\forall \vecv \forall \rho\in (0,1],  & \label{lebesguediff} \\
\lim_{\eta\rightarrow \rho^{-}} \frac{\int_\eta^\rho \hat{f}_1(u,\vecv)du}{\rho-\eta}=
\lim_{\eta\rightarrow \rho^{-}} \frac{\int_\eta^\rho
  \hat{f}_2(u,\vecv)du}{\rho-\eta} \ .\nonumber 
\end{align}

An estimator $\hat{f}$ is
 {\em nonnegative}  if $\forall S,\ \hat{f}(S)\geq 0$ and is
{\em unbiased} if $\forall  \vecv,\ \E[{\hat f} | \vecv]=f(\vecv)$.
An estimator has
{\em finite variance} on $\vecv$ if $\int_0^1 \hat{f}(u,\vecv)^2
du < \infty$ (the expectation of the square is finite) and is  {\em bounded} on $\vecv$ if $\sup_{u\in (0,1]}
  \hat{f}(u,\vecv) < \infty$.  If a nonnegative estimator is bounded
  on $\vecv$, it also
  has finite variance for $\vecv$.
An estimator is {\em monotone} on $\vecv$ if
when fixing $\vecv$ and considering outcomes consistent with $\vecv$, the estimate value is non decreasing with the
information on the data that we can glean from the outcome, that is, 
$\hat{f}(u,\vecv)$ is non-increasing with $u$.
We say that an estimator is bounded, has finite variances, or is
monotone, if the respective property holds for all $\vecv\in {\bf V}$.


\smallskip
\noindent
{\bf The lower bound function.}
For $Z\subset {\bf V}$,  we
define $\underline{f}(Z) = \inf\{ f(v) \mid v\in Z\}$ as
the infimum of $f$ on $Z$.
We use the notation $\underline{f}(S)\equiv \underline{f}(S^*)$,
$\underline{f}(\rho,\vecv) \equiv \underline{f}(S^*(\rho,\vecv))$.
When $\vecv$ is fixed, we use
$\underline{f}^{(\vecv)}(u) \equiv \underline{f}(u,\vecv)$.
Some  properties which we need in the sequel are~\cite{CKsharedseed:2012}:
\begin{align}  
\bullet & \text{$\forall \vecv$, 
$\underline{f}^{(\vecv)}(u)$ is } 
\text{monotone non increasing
and  left-continuous.}\label{proplb}  
& \\
\bullet & \text{$\hat{f}$ is unbiased and nonnegative} \implies  \label{cnonneg:eq} \\
& \forall \vecv, \forall \rho, \int_\rho^1 \hat{f}(u,\vecv)du \leq
\underline{f}^{(\vecv)}(\rho) \ .\label{nonneg:eq} 
\end{align}

The lower bound function $\underline{f}^{(\vecv)}$,  and its lower hull
$H^{(\vecv)}_f$, are instrumental in capturing existence
of estimators with desirable properties~\cite{CKsharedseed:2012}:
\begin{align}
\bullet & \text{$\exists$  unbiased nonnegative $f$ estimator}
 \iff \label{nec_reqa}\\
& \forall \vecv\in {\bf V},\,
\lim_{u \rightarrow 0^+} \underline{f}^{(\vecv)}(u) = f(\vecv)\
. \label{nec_req}\\
\bullet & \text{If $f$ satisfies \eqref{nec_req},} \nonumber\\
& \text{$\exists$ unbiased nonnegative estimator with finite variance for $\vecv$} \nonumber\\
&  \iff \int_0^1 \bigg(\frac{d H^{(\vecv)}_f(u)}{du} \bigg)^2 du  < \infty \ .  \label{bounded_var_shared_nec_req}\\
 & \text{$\exists$ unbiased nonnegative estimator that is bounded on $\vecv$} 
 \nonumber\\
 & \iff \lim_{u \rightarrow 0^+} \frac{f(\vecv)-\underline{f}^{(\vecv)}(u)}{u} < \infty \ . \label{bounded_shared_nec_req}
\end{align}
Example \ref{example3} illustrates lower bound functions and respective
lower hulls for $\range_{p+}$.

\medskip
\noindent
{\bf Partially specified estimators.}
We use {\em partial specifications} $\hat{f}$ of (nonnegative and unbiased)
estimators, which are
specified on a set of
outcomes ${\cal S}$ so that
\begin{align*}
\forall \vecv\ \exists\rho_v\in [0,1],\ &
S(u,\vecv)\in {\cal S}\, \mbox{{\it almost everywhere} for }  u>\rho_v\, \wedge\, \\
& S(u,\vecv)\not\in {\cal S}\, \mbox{{\it almost everywhere} for }  u\leq \rho_v\ .
\end{align*}
When $\rho_v=0$, we say that the estimator is
{\em fully specified} for $\vecv$.
We also require that $\hat{f}$ is
nonnegative where specified and satisfies
\begin{subequations}\label{pdnereq}
\begin{eqnarray}
\forall \vecv,\ \rho_v>0 & \implies & \int_{\rho_v}^1 \hat{f}(u,\vecv)du \leq f(\vecv) \label{condnonneg}\\
\forall \vecv,\ \rho_v=0 & \implies & \int_{\rho_v}^1
\hat{f}(u,\vecv)du = f(\vecv)\ . \label{condunbiased}
\end{eqnarray}
\end{subequations}

\begin{lemma}  \label{completioncoro} \cite{CKsharedseed:2012}
If $f$ satisfies \eqref{nec_req} (has a nonnegative unbiased
estimator), then any  partially specified estimator
can be extended to an unbiased nonnegative estimator.
\end{lemma}

\medskip
\noindent
{\bf $\vecv$-optimal extensions and estimators.}
Given a partially specified estimator $\hat{f}$ so that $\rho_v>0$ and
$M=\int_{\rho_v}^1 \hat{f}(u,\vecv)du$, a {\em $\vecv$-optimal extension} is
an extension which is fully specified for $\vecv$ and minimizes
variance for $\vecv$
(amongst all such extensions).   The $\vecv$-optimal  extension is
defined on outcomes $S(u,\vecv)$ for $u\in (0,\rho_v]$ and satisfies
\begin{align}
\text{min}\,  & \int_0^{\rho_v} \hat{f}(u,\vecv)^2 du \label{minvarcon} \\
 \text{s.t.} & \int_0^{\rho_v} \hat{f}(u,\vecv) du= f(\vecv)-M \nonumber \\
 & \forall u,  \int_u^{\rho_v} \hat{f}(x,\vecv)dx \leq \underline{f}^{(\vecv)}(u)-M \nonumber \\
 & \forall u,  \hat{f}(u,\vecv) \geq 0 \nonumber
\end{align}

 For $\rho_v\in (0,1]$ and $M\in [0,\underline{f}^{(\vecv)}(\rho_v)]$, we
define the function
$\hat{f}^{(\vecv,\rho_v,M)}:(0,\rho_v]\rightarrow R_+$ as the solution of
\begin{equation}
\hat{f}^{(\vecv,\rho_v,M)}(u) =\inf_{0\leq \eta < u}
\frac{\underline{f}^{(\vecv)}(\eta)-M-\int_u^{\rho_v}\hat{f}^{(\vecv,\rho_v,M)}(u)du}{\rho-\eta}\ . 
\end{equation}

Geometrically, the function $\hat{f}^{(\vecv,\rho_v,M)}$  is the negated derivative of the lower hull of the lower bound function $\underline{f}^{(\vecv)}$ on $(0,\rho_v)$ and the point $(\rho_v,M)$.

\begin{theorem}  \label{voptlh} \cite{CKsharedseed:2012}
Given a partially specified
estimator $\hat{f}$ so that $\rho_v>0$ and
$M=\int_{\rho_v}^1 \hat{f}(u,\vecv)du$, then
$\hat{f}^{(\vecv,\rho_v,M)}$ is the unique (up to equivalence)
$\vecv$-optimal extension of $\hat{f}$.
\end{theorem}

The {\em $\vecv$-optimal} estimates are the minimum variance extension of
the empty specification. We use $\rho_v=1$ and $M=0$ and obtain
 $\hat{f}^{(\vecv)}\equiv \hat{f}^{(\vecv,1,0)}$.
$\hat{f}^{(\vecv)}$ is the solution of
\begin{equation}  \label{PWopt}
\hat{f}^{(\vecv)}(u) =\inf_{0\leq \eta < u}
\frac{\underline{f}^{(\vecv)}(\eta)-\int_u^1 \hat{f}^{(\vecv)}(u)du}{\rho-\eta}\ ,
\end{equation}
which is the negated slope of the lower hull of the lower bound function
$\underline{f}^{(\vecv)}$.
This is illustrated in Example \ref{example3}.

\smallskip
\noindent
{\bf Admissibility and order optimality.}
  An estimator is  {\em admissible} if there is no
  (nonnegative unbiased) estimator with same or lower variance on all data
and strictly lower on some data.
We also consider {\em order optimality}, specified  with respect to
a partial order $\prec$ on ${\bf V}$:
An estimator $\hat{f}$ is
{\em $\prec^{+}$-optimal}
 if there is no other nonnegative unbiased estimator with
strictly lower variance on some data $\vecv$ and at most the
variance of  $\hat{f}$ on all vectors that precede $\vecv$.
Order-optimality (with respect to some $\prec$) implies admissibility
but the converse is not true in general~\cite{CK:pods11}.

\smallskip
\noindent
{\bf Variance competitiveness}
An estimator $\hat{f}$ is {\em $c$-competitive} if
$$\forall \vecv,\, \int_0^1 \bigg(\hat{f}(u,\vecv)\bigg)^2du \leq c \inf_{\hat{f}'}\int_0^1 \bigg(\hat{f}'(u,\vecv)\bigg)^2du ,$$
where the infimum is over all unbiased nonnegative estimators of $f$.
When the estimator is unbiased, the expectation of the
square is closely related to variance, and an estimator that minimizes
one also minimizes the other.
\begin{align}
\var[\hat{f} | \vecv] 
&=\int_0^1 \hat{f}(u,\vecv)^2 du -f(\vecv)^2 \label{var2moment}
\end{align}

\section{The optimal range}  \label{poptr:sec}

\begin{figure}[htbp]
\centerline{
\ifpdf
\includegraphics[width=0.44\textwidth]{POrange}
\else
 \epsfig{figure=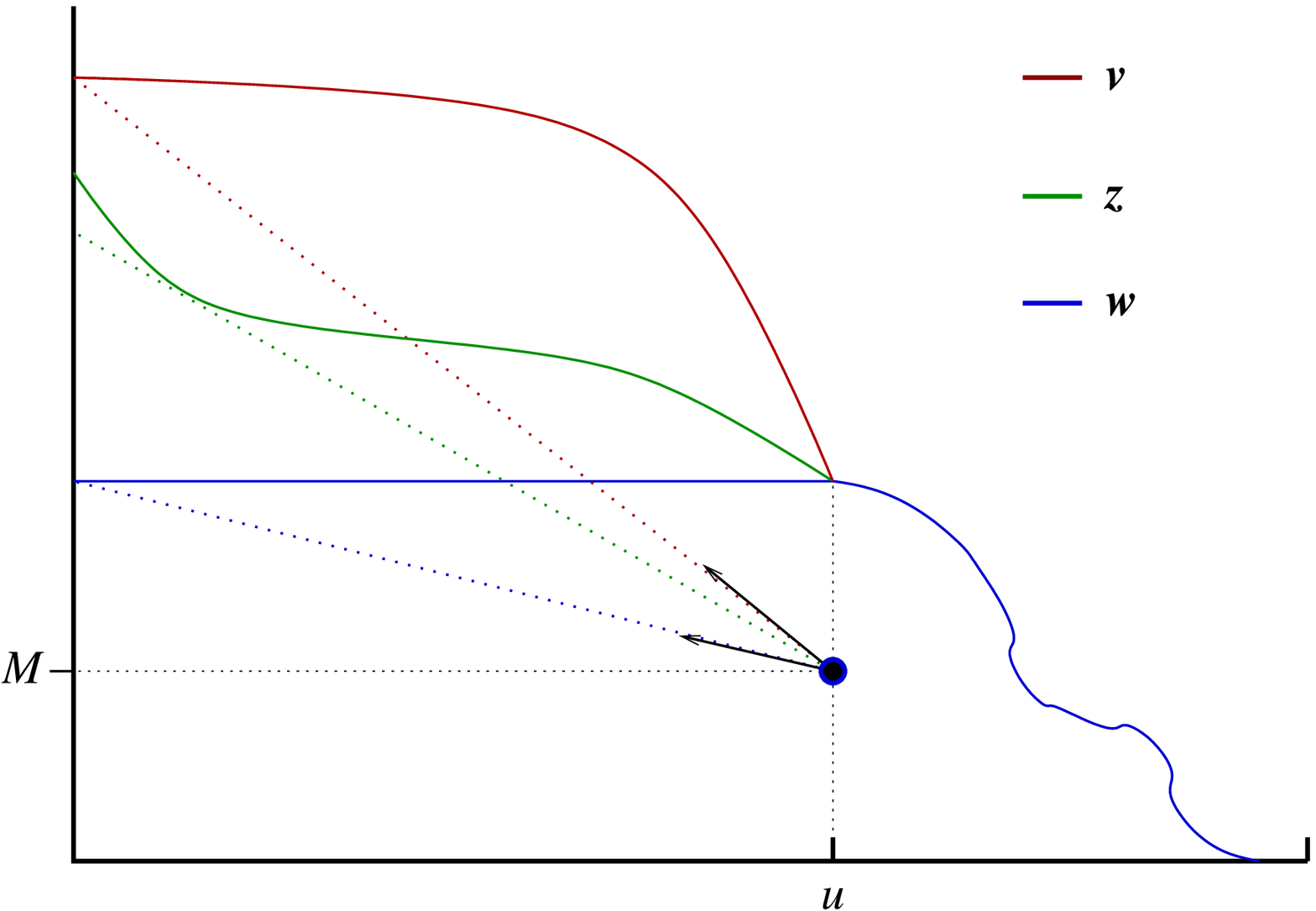,width=0.44\textwidth}
\fi
}
\caption{Lower bound functions for vectors $\vecv,\vecz,\vecw$.  Outcomes are consistent
for all $x\geq u$: $S(x,\vecv)=S(x,\vecz)=S(x,\vecw)\equiv S_x$.  The
figure illustrates the $\vecy$-optimal estimates $\lambda(u,\vecy,M)$ at $u$ given $M$ for 
${\bf y}\in \{\vecv,\vecz,\vecw\}$.
The estimates are the negated slopes of the lower hull of the point $(u,M)$
and the lower bound function $\underline{f}^{(\vecy)}$.
The optimal range at $S_u$ given $M$ is lower-bounded by $\vecw$, that is
$\lambda_L(S_u,M)=\lambda(u,\vecw,M)$,  and upper-bounded by $\vecv$, 
$\lambda_U(S_u,M)=\lambda(u,\vecv,M)$.  The figure illustrates the
general property  that
the optimal range is lower bounded by the $\vecw$ which satisfies
$f(\vecw)= \underline{f}(\vecw,u)$.
 \label{POrange:fig}}
\end{figure}

We say that an estimator 
$\hat{f}$ is
$\vecv$-optimal {\em at an outcome $S(u,\vecv)$}
 if it satisfies \eqref{PWopt}.  
For an outcome $S(\rho,\vecv)$, we are interested in the {\em range} of
 $\vecz$-optimal estimates at $S$ for all
$\vecz\in S^*$,  with respect to a value $M$, which captures 
the contribution to the expectation of the estimator made by outcomes
which are less informative than $S$.

\begin{align}
\lambda(\rho,\vecv,M) &=\inf_{0\leq \eta < \rho}
\frac{\underline{f}(\eta,\vecv)-M}{\rho-\eta}\ \label{lambdaMdef}\\
\lambda_U(\rho,\vecv,M) 
&\equiv \lambda_U(S,M)
=\sup_{\vecz\in S^*(\rho,\vecv)}
\lambda(\rho,\vecz,M)\ \label{lambdaULdef}\\
\lambda_L(\rho,\vecv,M) 
&\equiv \lambda_L(S,M)
=\inf_{\vecz\in S^*(\rho,\vecv)}
\lambda(\rho,\vecz,M)  \nonumber  \\
&= \inf_{\vecz\in S^*(\rho,\vecv)} \inf_{0\leq \eta < \rho}
\frac{\underline{f}(\eta,\vecz)-M}{\rho-\eta}\  \nonumber \\ &
=\frac{\underline{f}(\rho,\vecv)-M}{\rho}\label{lambdaMLdef}
\end{align}
To verify equality  \eqref{lambdaMLdef}, observe that from
left continuity of $\underline{f}(u,\vecz)$,
$$\inf_{\eta<\rho,\ \vecz\in S^*} \underline{f}(\eta,\vecz)=\underline{f}(\rho,\vecv)$$
and that the denominator $\rho-\eta$ is maximized at $\eta=0$.
$\lambda(\rho,\vecv,M)$ is the $\vecv$-optimal estimate
at $\rho$, given a specification of the estimator $\hat{f}(u,\vecv)$ for $u\in
(\rho,1]$
with $\int_{\rho}^1 \hat{f}(u,\vecv)du = M$.  
In short, we refer to $\lambda(\rho,\vecv,M)$ as the $\vecv$-optimal
estimate at $\rho$ given $M$.
Geometrically, $\lambda(\rho,\vecv,M)$ is the negated slope
of the lower hull of $\underline{f}^{(\vecv)}$ and the point
$(\rho,M)$.
$\lambda_U(S,M)$ and
$\lambda_L(S,M)$, respectively, are the supremum and
infimum of the {\em range} of $\vecz$-optimal estimates at $S$ given
$M$.
Figure~\ref{POrange:fig} illustrates an outcome $S$ and the optimal
range at $S$ given $M$.
We can see how the lower endpoint of
the range is realized by a vector with $f$ value equal to the lower
bound at $S$, as in equality  \eqref{lambdaMLdef}.

 When $\hat{f}$ is provided for seed values $u\in (\rho,1]$,  we use $M=\int_{\rho}^1
\hat{f}(u,\vecv)du$.  We then abbreviate
the notations (we remove $M$) to
$\lambda(\rho,\vecv)$, $\lambda_U(S)$, and
$\lambda_L(S)$.


We say that the estimator $\hat{f}$ 
is {\em in-range} (in the
optimal range ) at outcome $S(\rho,\vecv)$ if
\begin{equation} \label{inporange}
\lambda_L(S) \leq \hat{f}(S)\leq
\lambda_U(S)\ .
\end{equation}

 Writing \eqref{inporange} explicitly, we obtain
{\small
\begin{subequations}\label{necessarydomshared}
\begin{eqnarray}
\hat{f}(\rho,\vecv) & \geq &  \lambda_L(\rho,\vecv) 
= \frac{\underline{f}(\rho,\vecv)- \int_{\rho}^1 \hat{f}(u,\vecv) du}{\rho}  \label{b1:ineq}\\
\hat{f}(\rho,\vecv) & \leq &  \lambda_U (\rho,\vecv)  \nonumber\\ &=&  \sup_{\vecz\in S^*}\inf_{0\leq \eta < \rho} \frac{\underline{f}(\eta,\vecz)-\int_{\rho}^1 \hat{f}(u,\vecv) du}{\rho-\eta} \label{b3:ineq}
\end{eqnarray}
\end{subequations}
}

Two special solutions  
that we study are the {\em \L\ estimator} ($\hat{f}^{(L)}$, see Section~\ref{Lest:sec})
and the {\em \U\ estimator} ($\hat{f}^{(U)}$, see Section~\ref{Uest:sec}), which respectively solve
\eqref{b1:ineq} and \eqref{b3:ineq} with equalities.
For all $\rho\in (0,1]$ and $\vecv$, $\hat{f}^{(L)}$ minimizes and
$\hat{f}^{(U)}$ maximizes $\int_{\rho}^1 \hat{f}(u,\vecv)du$ among
all solutions of \eqref{inporange}.
\ignore{$\hat{f}$ is equivalent to
a solution of \eqref{necessarydomshared} if and only if
\begin{subequations}\label{altnds}
\begin{eqnarray}
\lefteqn{\forall \rho\in (0,1]\ \forall \vecv,} \nonumber\\
&& \int_{\rho}^1 \hat{f}(u,\vecv)du  \geq \int_{\rho}^1 \hat{f}^{(L)}(u,\vecv)du \label{altnds1} \\
&& \int_{\rho}^1 \hat{f}(u,\vecv)du \leq \int_{\rho}^1 \hat{f}^{(U)}(u,\vecv)du\ . \label{altnds3}
\end{eqnarray}
\end{subequations}
}


 We show that being in-range (satisfying \eqref{inporange} for all
 outcomes $S$) is sufficient for nonnegativity and unbiasedness.
\begin{lemma} \label{rangeunon}
If $f$ satisfies \eqref{nec_req}
then any in-range  estimator 
is unbiased and nonnegative.
\end{lemma}
\begin{proof}
For nonnegativity, it
suffices to show that a solution of \eqref{inporange}  satisfies \eqref{nonneg:eq}, since 
\eqref{b1:ineq} and \eqref{nonneg:eq} together imply nonnegativity.
 Assume to the contrary that a solution $\hat{f}$ violates \eqref{nonneg:eq} and
 let $\rho$ be the supremum of $x$ satisfying
$\int_x^1 \hat{f}(u,\vecv)du > \underline{f}(x,\vecv)$.
From \eqref{proplb},  which is monotonicity and left-continuity of $\underline{f}(x,\vecv)$, we have
$\int_\rho ^1 \hat{f}(u,\vecv)du = \underline{f}(\rho,\vecv)$.
Since $\int_x^1 \hat{f}(u,\vecv) du$ is continuous in $x$, and
$\underline{f}^{(\vecv)}$ left-continuous, there must be $\delta>0$
so that 
\begin{equation}\label{tocontra}
\forall x\in[\rho-\delta,\rho), \int_x^1 \hat{f}(u,\vecv)du >
\underline{f}(x,\vecv)\ .
\end{equation}

Let
 $x\in[\rho-\delta,\rho)$ and
$M(x)= \int_x^1 \hat{f}(u,\vecv)du$. From \eqref{tocontra},  $M(x)>\underline{f}(x,\vecv)$.
We have that 
\begin{eqnarray*}
\hat{f}(x,\vecv) &\leq & \sup_{\vecz\in S^*(x,\vecv)}
\inf_{0\leq
  \eta<x}\frac{\underline{f}(\eta,\vecz)-M(x)}{x-\eta}
\\
&\leq & \sup_{\vecz\in S^*(x,\vecv)}
\inf_{0\leq \eta<x}\frac{\underline{f}(\eta,\vecz)-\underline{f}(x,\vecv)}{x-\eta}  \\
& \leq & \sup_{\vecz\in S^*(x,\vecv)}
\lim_{\eta\rightarrow x^{-}} \frac{\underline{f}(\eta,\vecz)-\underline{f}(x,\vecv)}{x-\eta} \\
& = & \lim_{\eta\rightarrow x^{-}} \frac{\underline{f}(\eta,\vecv)-\underline{f}(x,\vecv)}{x-\eta} =  -\frac{\partial \underline{f}(x,\vecv)}{\partial x^{-}}
\end{eqnarray*}

Since this holds for all $x\in(\rho-\delta,\rho)$, we obtain that 
$\int_{\rho-\delta}^\rho \hat{f}(x,\vecv)dx \leq
\underline{f}(\rho-\delta,\vecv)-\underline{f}(\rho,\vecv)$.
Therefore,
$\int_{\rho-\delta}^1 \hat{f}(x,\vecv)dx \leq
\underline{f}(\rho-\delta,\vecv)$, 
which contradicts \eqref{tocontra}.

\ignore{
From \eqref{proplb},  which is monotonicity and left-continuity of $\underline{f}(x,\vecv)$, we have
$\int_\rho ^1 \hat{f}(u,\vecv)du = \underline{f}(\rho,\vecv) \equiv
M$ and
$$\lim_{x\rightarrow \rho^1} \frac{\int_x^\rho \hat{f}(u,\vecv)}{x-\rho} > \lim_{x\rightarrow \rho^{-}} \frac{\underline{f}(x,\vecv)-\underline{f}(\rho,\vecv)}{\rho-x}\ .$$

Let $M=\int_\rho^1 \hat{f}(u,\vecv)du$.  Consider the $\vecz$-optimal
extension
$\hat{f}^{(\vecz,\rho,M)}$ for 
 vectors $\vecz\in S(\rho,\vecv)$,
all these extensions, and thus their supremum satisfy
$$\hat{f}^{(\vecz,\rho,M)}(\rho) \leq \lim_{x\rightarrow \rho^{-}} \frac{\underline{f}(x,\vecv)-\underline{f}(\rho,\vecv)}{\rho-x}$$
and we obtain a contradiction.
\begin{eqnarray*}
\hat{f}(\rho,\vecv) &\leq & \sup_{\vecz\in S^*(\rho,\vecv)}
\inf_{0\leq \eta<\rho}\frac{\underline{f}(\eta,\vecz)-\underline{f}(\rho,\vecv)}{\rho-\eta}  \\
& \leq & \sup_{\vecz\in S^*(\rho,\vecv)}
\lim_{\eta\rightarrow \rho^{-}} \frac{\underline{f}(\eta,\vecz)-\underline{f}(\rho,\vecv)}{\rho-\eta} \\
& = & \lim_{\eta\rightarrow \rho^{-}} \frac{\underline{f}(\eta,\vecv)-\underline{f}(\rho,\vecv)}{\rho-\eta} =  -\frac{\partial \underline{f}(\rho,\vecv)}{\partial \rho^{-}}
\end{eqnarray*}
}

 We now establish unbiasedness.  From \eqref{b1:ineq}  and
$\underline{f}(u,\vecv)$ being non increasing in $u$, we  obtain that
$\forall u \forall \rho>u$,
\begin{eqnarray}
\hat{f}(u,\vecv) & \geq & \frac{\underline{f}(u,\vecv)-\int_u^1 \hat{f}(x,\vecv)dx}{u}\nonumber\\
& \geq & \frac{\underline{f}(\rho,\vecv)-\int_u^1
  \hat{f}(x,\vecv)dx}{u} \label{fraccond}
\end{eqnarray}
We argue that
\begin{equation}\label{claimub}
\forall \vecv \forall \rho>0,\
\lim_{x\rightarrow 0} \int_x^1 \hat{f}(u,\vecv)du \geq \underline{f}(\rho,\vecv)\ .
\end{equation}
To prove  \eqref{claimub},
define $\Delta(x)= \underline{f}(\rho,\vecv)-\int_x^1
\hat{f}(u,\vecv)du$ for $x\in (0,\rho]$.  We show that
 $\int_{x/2}^x \hat{f}(u,\vecv) du \geq \Delta(x)/4$. To
 see this, assume to the contrary that  $\int_y^x \hat{f}(u,\vecv)du\leq
 \Delta(x)/4$ for all $y\in [x/2,x]$.  Then  from
 \eqref{fraccond},  the value of $\hat{f}(u,\vecv)$ for
 $u\in [x/2,x]$ must be at least
$(3/4)\Delta(x)/x$.  Hence, the integral  over the interval
$[x/2,x]$ is at
least $(3/8)\Delta(x)$ which is a contradiction.  We can now apply
this iteratively, obtaining that $\Delta(\rho/2^i) \leq (3/4)^i
\Delta(\rho)$.
Thus, the gap $\Delta(x)$ diminishes as $x\rightarrow 0$ and we
established \eqref{claimub}.

Since \eqref{claimub} holds for all $\rho \geq 0$, then
$\lim_{u\rightarrow 0} \int_u^1 \hat{f}(u,\vecv)du \geq \lim_{u\rightarrow 0}\underline{f}(u,\vecv)= f(\vecv)$
(using \eqref{nec_req}).  Combining with (already established)
\eqref{nonneg:eq} we obtain
$\lim_{u\rightarrow 0} \int_u^1 \hat{f}(u,\vecv)du = f(\vecv)$.
\end{proof}

  We next show that being in-range is necessary for optimality.
For our analysis of order-optimality
(Section~\ref{estPREC:sec}), we need to 
slightly refine the notion of
admissibility to be with respect to
a partially specified estimator $\hat{f}$ and a
subset of data vectors $Z\subset {\bf V}$.

An extension of $\hat{f}$  that is
fully specified for all vectors in $Z$ is 
 admissible on $Z$ if
any other extension with strictly lower variance on at least one $\vecv\in Z$ has a strictly higher variance on at least one $\vecz\in Z$.
  We say that a partial specification is in-range {\em with respect to $Z$} if:
\begin{eqnarray}
\forall \vecv\in Z, \text{ for }  \rho\in (0,\rho_v] \text{ almost
  everywhere, }\quad &&\nonumber\\
 \inf_{\vecz\in Z\cap S^*(\rho,\vecv)} \lambda(\rho,\vecz) \leq  \hat{f}(\rho,\vecv)  \leq   \sup_{\vecz\in Z\cap
S^*(\rho,\vecv)} \lambda(\rho,\vecz) && \label{necessary}
\end{eqnarray}
Using \eqref{lebesguediff}, \eqref{necessary} is the same as
requiring that $\forall \vecv\ \forall \rho\in (0,\rho_v]$, when
fixing the estimator on $S(u,\vecv)$ for $u\geq \rho$, then
{\small
\begin{align}
 \inf_{\vecz\in Z\cap S^*(\rho,\vecv)}
\lambda(\rho,\vecz) &\leq 
\lim_{\eta\rightarrow \rho^-}\frac{\int_\eta^{\rho} \hat{f}(u,\vecv)du}{\rho-\eta} 
\leq   \sup_{\vecz\in Z\cap S^*(\rho,\vecv)} \lambda(\rho,\vecz) \label{inecessary}
\end{align}
}

We show that a necessary condition for admissibility with
respect to a partial specification and $Z$ is that almost everywhere,
estimates for outcomes consistent with vectors in $Z$ are in-range
for $Z$.  Formally: 
\begin{theorem} \label{domshared:thm}  \label{rvarplus}
An extension is admissible on $Z$ only if \eqref{necessary}
holds.
\end{theorem}
\begin{proof}
Consider  an (nonnegative unbiased) estimator $\hat{f}$ that violates \eqref{necessary} for some $\vecv\in Z$ and $\rho$.  We show that there is an alternative estimator,
equal to $\hat{f}(u,\vecv)$ on outcomes $u>\rho$ and which satisfies
\eqref{necessary} at $\rho$ that has strictly lower variance than
$\hat{f}$ on all vectors $Z\cap S^*(\rho,\vecv)$.  This will show
that $\hat{f}$ is not admissible on $Z$.

The estimator $\hat{f}$ violates \eqref{inecessary}, so either
\begin{equation}\label{svio}
\lim_{\eta\rightarrow \rho^-}\frac{\int_\eta^{\rho} \hat{f}(u,\vecv)du}{\rho-\eta} <  \inf_{\vecz\in Z\cap S^*(\rho,\vecv)} \lambda(\rho,\vecz)\equiv L
\end{equation}
or
\begin{equation}\label{lvio}
\lim_{\eta\rightarrow \rho^-}\frac{\int_\eta^{\rho} \hat{f}(u,\vecv)du}{\rho-\eta} >  \sup_{\vecz\in Z\cap S^*(\rho,\vecv)}
\lambda(\rho,\vecz) \equiv U\ .
\end{equation}
Violation \eqref{lvio}, for a nonnegative unbiased $\hat{f}$,  means
that 
$M\equiv \int_{\rho}^1 \hat{f}(u,\vecv)du <
\underline{f}(u,\vecv)$.  Consider
$\vecz\in Z\cap
S^*(\rho,\vecv)$ and
the $\vecz$-optimal extension,
$\hat{f}^{(\vecz,\rho,M)}$ (see Theorem~\ref{voptlh}).
Because the point $(\rho,M)$ lies strictly below
$\underline{f}^{(\vecz)}$,  the lower hull of both the point and
$\underline{f}^{(\vecz)}$  has a
linear piece on some interval with right end point $\rho$.  More precisely,
$\hat{f}^{(\vecz,\rho,M)}(u) \equiv \lambda(\rho,\vecz,M)$ on $S(u,\vecz)$ at
some nonempty interval $u\in (\eta_z,\rho]$ so that at the point $\eta_z$, the
lower bound is met, that is, $M +
(\rho-\eta_z) \lambda(\rho,\vecz,M) = \lim_{u\rightarrow \eta_z^+}
\underline{f}(u,\vecz)$.  Therefore, all extensions (maintaining
nonnegativity and  unbiasedness) must
satisfy
\begin{align} 
\int_{\eta_z}^{\rho} \hat{f}(u,\vecz)du  & \leq \lim_{u\rightarrow \eta_z^+}
\underline{f}(u,\vecz)-M  \label{mustnn} \\
& = (\rho-\eta_z)
\lambda(\rho,\vecz,M) \leq (\rho-\eta_z)U\ .\nonumber
\end{align}
From \eqref{lvio},
for some $\epsilon>0$, $\hat{f}$ has average value strictly higher than $U$
on $S(u,\vecv)$ for all $u$ in $(\eta,\rho]$ for $\eta\in
[\rho-\epsilon,\rho)$.  For each $\vecz\in S^*(\rho,\vecv)$ we
define $\zeta_z$ as the maximum of $\rho-\epsilon$ and $\inf\{u \mid S^*(u,\vecv)=S^*(u,\vecz)\}$.
From \eqref{openset:lemma}, 
$\zeta_z< \rho$.  For each $\vecz$, the higher estimate values
on $S(u,\vecz)$ for $u\in (\zeta_z,\rho]$
must be
``compensated for'' by lower values on $u\in (\eta_z,\zeta_z)$
(from nonnegativity we must have the strict inequality $\eta_z < \zeta_z$)
so that \eqref{mustnn} holds.  By modifying the estimator to be equal to
$U$ for all outcomes $S(u,\vecv)$ $u\in (\rho-\epsilon,\rho]$ and
correspondingly increasing some estimate  values that are lower than
$U$ to $U$ on $S(u,\vecz)$ for $u\in (\eta_z,\zeta_z)$ we obtain an estimator
with strictly lower variance than $\hat{f}$ for all $\vecz\in Z \cap
S^*(\rho,\vecv)$ and same variance as $\hat{f}$ on all other vectors.
Note we can perform the shift consistently across all   branches of
the tree-like partial order on outcomes.

Violation \eqref{svio} means that
for some $\epsilon>0$, $\hat{f}$ has average value strictly lower than $L$
on $S(u,\vecv)$ for all intervals $u\in (\eta,\rho]$ for $\eta\in
[\rho-\epsilon,\rho)$.
For all $\vecz$, the $\vecz$-optimal extension
$\hat{f}^{(\vecz,\rho,M)}(u)$  has value $\lambda(\rho,\vecz,M)\geq L$
at $\rho$ and (from convexity of lower hull)  values that are at least that on $u<\rho$.
From unbiasedness,  we must have for all $\vecz\in Z\cap
S^*(\rho,\vecv)$, $\int_0^\rho \hat{f}(u,\vecz)du=\int_0^\rho
\hat{f}^{(\vecz,\rho,M)}(u) du$.  Therefore, values lower than $L$
must be compensated for in $\hat{f}$ by values higher than $L$.
We can modify the estimator such that it is equal to $L$ for $S(u,\vecv)$ for $u\in (\rho-\epsilon,\rho)$ and compensate for that by
lowering values at lower $u$ values $u<\zeta_z$ that are higher than $L$.
The modified estimator
has strictly lower variance than $\hat{f}$ for all $\vecz\in Z \cap
S^*(\rho,\vecv)$ and same variance as $\hat{f}$ on all other vectors.
\end{proof}

\section{The \L\ Estimator} \label{Lest:sec}

The \L\ estimator, 
$\hat{f}^{(L)}$, is the solution of \eqref{b1:ineq} with equalities, obtaining
values that are minimum in the optimal range.  
Formally, it is the solution of the
integral equation $\forall \vecv\in {\bf V}$, $\forall \rho\in (0,1]$:
\begin{eqnarray}
\hat{f}^{(L)}(\rho,\vecv) 
& =& \frac{\underline{f}^{(\vecv)}(\rho)-\int_{\rho}^1 \hat{f}^{(L)}(u,\vecv) d u }{\rho} \label{LBdefeq}
\end{eqnarray}
Geometrically,  as visualized in Figure~\ref{Lsteps:fig},
the \L\ estimate on an outcome $S(\rho,\vecv)$  is exactly the slope
value that if maintained for outcomes $S(u,\vecv)$ ($u\in
(0,\rho]$), would yield an expected estimate of $\underline{f}(S)$. 
\begin{figure}[htbp]
\centerline{
\ifpdf
\includegraphics[width=0.44\textwidth]{Lsteps_nolatex}
\else
 \epsfig{figure=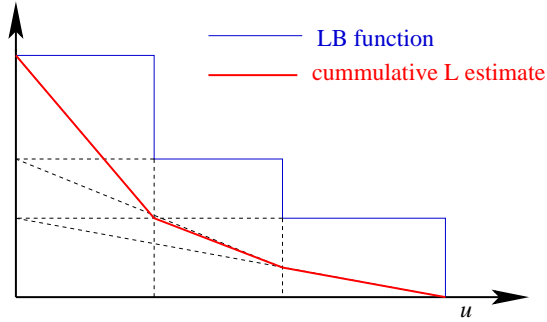,width=0.44\textwidth}
\fi
}
\caption{An example lower bound function $\underline{f}^{(\vecv)}(u)$ with 3 steps and the respective cummulative L estimate
  $\int_u^1 \hat{f}^{(L)}(u,\vecv)du$.  The estimate $\hat{f}^{(L)}$
  is the   negated slope and in this case is also a step function with 3 steps. \label{Lsteps:fig}}
\end{figure}
We derive a
convenient
expression for the \L\ estimator, which enables us to derive explicit
forms or compute it for any function $f$.
We show that the \L\ estimator
is $4$-competitive and that it is the unique admissible monotone
estimator.
We also show it is order-optimal with respect to
the natural order that prioritizes data vectors with lower $f(\vecv)$.

Fixing $\vecv$, \eqref{LBdefeq} is a first-order differential equation for $F(\rho)
\equiv \int_\rho^1 \hat{f}^{(L)}(u,\vecv) d u $ and the initial
condition $F(1)=0$.  Since the lower
bound function $\underline{f}^{(\vecv)}$ is monotonic and bounded, it
is continuous (and differentiable) almost everywhere.  Therefore, the
equation with the initial condition has a unique solution:




\begin{lemma}  \label{Lestimator}
\begin{align}
\hat{f}^{(L)}(\rho,\vecv) &= \frac{\underline{f}^{(\vecv)}(\rho)}{\rho}- \int_{\rho}^1
\frac{\underline{f}^{(\vecv)}(u)}{u^2} du \label{LBgen} \\
\end{align}
When $\underline{f}^{(\vecv)}(1)=0$, which we can assume without loss
of generality\footnote{Otherwise, we can instead estimate
the function  $f(\vecv)-\underline{f}^{(\vecv)}(1)$, which satisfies
this assumption,  and then add a fixed value of
  $\underline{f}^{(\vecv)}(1)$ to the resulting estimate.}, the solution has the simpler form:
\begin{align}
\hat{f}^{(L)}(\rho,\vecv) &= -\int_\rho^1 \frac{\frac{d \underline{f}^{(\vecv)}(u)}{d u}}{u}du \label{shortLB}
\end{align}
\end{lemma}


 We show a tight bound of $4$ for the competitive ratio for
$\hat{f}^{(L)}$, meaning that it is at most $4$ for all functions $f$ and
for any $\epsilon>0$, 
there exists a function $f$ on which the ratio is no less than
$4-\epsilon$.
\begin{theorem}  \label{Ltight4:thm}
\begin{align*}
& \sup_{f,\vecv\, |  \int_0^1 \hat{f}^{(\vecv)}(u)^2du< \infty}
\frac{\int_0^1 \hat{f}^{(L)}(u,\vecv)^2du}{\int_0^1 \hat{f}^{(\vecv)}(u)^2du} = 4 \ ,
\end{align*}
\end{theorem}

 We present a family of functions for which the supermum of this
 ratio is $4$.  We use the domain
${\bf V}=[0,1]$,  a PPS sampling scheme with $\tau(u)=u$, and the
function
$f(v)=\frac{1}{1-\p}-\frac{v^{1-\p}}{1-\p}$
for $\p \in [0,0.5)$. For the data $v=0$ we have the
following convex lower bound function
$$\underline{f}(u,0)=\frac{1}{1-\p}-\frac{u^{1-\p}}{1-\p} \ .$$

Being convex, this lower bound function is equal to its lower hull.
Therefore, by taking its negated derivative, we get
 $\hat{f}^{(0)}(u)=1/u^{\p}$.
  The function $\hat{f}^{(0)}$ is
square integrable when $\p< 0.5$:
\begin{align*}
\int_0^1 \hat{f}^{(0)}(u)^2 du &=\int_0^1 1/u^{2\p}du=
\frac{1}{1-2\p}\ .
\end{align*}
From \eqref{shortLB},  the \L\ estimator on outcomes consistent with
$v=0$ for $\p\in (0,0.5)$ is\footnote{ For $\p=0$ the estimate
  is $-\ln(x)$.}
\begin{align*}
\hat{f}^{(L)}(x,0)&=\int_x^1 \frac{1}{u^{1+\p}}
=\frac{1}{\p}\bigg(\frac{1}{x^{\p}}-1\bigg)\ .
\end{align*}
Hence,
\begin{eqnarray*}
\lefteqn{\int_0^1 \hat{f}^{(L)}(u,0)^2 du =\frac{1}{\p^2}\int_0^1
\bigg(\frac{1}{u^{2\p}}-\frac{2}{u^{\p}}+1\bigg) du} \\ 
&=&
\frac{1}{\p^2}\bigg(\frac{1}{1-2\p}-\frac{2}{1-\p}+1\bigg)
=
\frac{2}{(1-2\p)(1-\p)}\ .
\end{eqnarray*}
We obtain  the ratio
$$\frac{\int_0^1 \hat{f}^{(L)}(u,0)^2 du}{\int_0^1 \hat{f}^{(0)}(u)^2 du}=\frac{2}{1-\p} \leq 4 \ .$$
The ratio approaches $4$ when $\p\rightarrow 0.5^{-}$.




We conclude the proof of Theorem \ref{Ltight4:thm} using 
the following lemma that shows that if  $\hat{f}^{(\vecv)}(u)$ is square integrable, that is, \eqref{bounded_var_shared_nec_req} holds, then $\hat{f}^{(L)}(u,\vecv)$ is also square integrable and the ratio between these integrals is at most $4$.
\begin{lemma} \label{Lestbound}
\begin{align*}
& \forall \vecv,\ \int_0^1 \hat{f}^{(\vecv)}(u)^2du < \infty \implies
\frac{\int_0^1 \hat{f}^{(L)}(u,\vecv)^2du}{\int_0^1
  \hat{f}^{(\vecv)}(u)^2du}\leq 4\ .
\end{align*}
\end{lemma}
\begin{proof}
Fixing $\vecv$, 
the function $\hat{f}^{(\vecv)}$ only depends on the lower hull
of the lower bound function $\underline{f}^{(\vecv)} (u)$.  The
estimator $\hat{f}^{(L)}$ depends on the  lower bound function
$\underline{f}$ and can be different for different lower bound
functions with the same lower hull.  Fixing the lower hull, the
variance of the \L\ estimator is maximized 
for $f$ such that $\underline{f}^{(\vecv)}\equiv H^{(\vecv)}_f$.
 It therefore suffices to
consider 
convex $\underline{f}^{(\vecv)}(u)$, that is, $\frac{d^2 \underline{f}^{(\vecv)}(u)}{d^2
u} > 0$ for which we have
$$\hat{f}^{(\vecv)}(u)=-\frac{d \underline{f}^{(\vecv)}(u)}{d u}\ .$$ 
Recall that $\hat{f}^{(\vecv)}(u)$ is monotone non-increasing.
From \eqref{shortLB},  $\hat{f}^{(L)}(\rho,\vecv) =
-\int_\rho^1 \frac{\frac{d \underline{f}^{(\vecv)}(u)}{d u}}{u}du$.

To establish our claim, it suffices to show that for all monotone, non
increasing, square integrable functions $g:(0,1]$,
\begin{equation} \label{toshow}
\frac{\int_0^1 (\int_x^1 \frac{g(u)}{u} du)^2 dx}{\int_0^1 g(x)^2
  dx}\leq 4
\end{equation}

Define $h(x)=\int_x^1 \frac{g(u)}{u} du$.
\begin{align*}
\int_{\epsilon}^{1} h^2(x) dx &=  \int_{\epsilon}^1 \int_x^1 2 h(y)
h'(y) dy dx  \\ &
= \int_{\epsilon}^1 \int_{\epsilon}^{y} 2 h(y) h'(y) dx dy  \\ &
= 2\int_{\epsilon}^1  h(y) h'(y) \int_{\epsilon}^{y} dx dy \\
&= 2\int_{\epsilon}^1  h(y) h'(y) (y-\epsilon) dy  \\ &
= 2\int_{\epsilon}^1  h(y) \frac{g(y)}{y} (y-\epsilon) dy \le
2\int_{\epsilon}^1  h(y) g(y)  dy \\
 &\le 2 \sqrt{\int_{\epsilon}^{1} h^2(y) dy} \sqrt{\int_{\epsilon}^{1}
g^2(y) dy}
\end{align*}

The last inequality is Cauchy-Schwartz.  To obtain \eqref{toshow}, we divide both
sides by $\sqrt{\int_{\epsilon}^{1} h^2(y) dy}$ and take the limit
as $\epsilon$ goes to $0$. 

\ignore{
We consider the Hilbert space $L^2[0,1]$ of of square integrable
functions on $[0,1]$.  The functions $\exp(i2\pi n x)$ for all integers
$n$ are an orthonormal basis of this space.
We can establish that for each basis element, $\exp(i2\pi
nu)$,
the function $\int_x^1 \frac{\exp(i2\pi nu)}{u} du$ is square
integrable, that is, the integral over $[0,1]$ of its product by its
complex conjugate is finite.  Moreover, it decreases with $|n|$.


{\bf *****  Check again:  the basis is countable not finite.  }
}
\ignore{
Using integration by parts
$\int \frac{g(u)}{u}du = g(u)\ln(u)-\int g'(u) \ln(u) du$.
Thus,
$\int_x^1 \frac{g(u)}{u}du = -g(x)\ln(x) -\int_x^1 g'(u) \ln(u)$.
Using, $0\leq -\int_x^1 g'(u) \ln(u) \leq -\ln(x)g(x)$
We obtain that $|\int_x^1 \frac{g(u)}{u}du|\leq -g(x)\ln(x)$.
If $g(x)$ is square integrable, so is $g(x)\ln(x)$ and we conclude.
}
\end{proof}

\begin{theorem}
The estimator $\hat{f}^{(L)}$ is monotone.
Moreover, it is the unique admissible  monotone estimator
and dominates all
monotone estimators.
\end{theorem}
\begin{proof}
Recall that an estimator $\hat{f}$ is monotone if and only if,
for any data $\vecv$, the
estimate $\hat{f}(\rho,\vecv)$ is non-increasing with $\rho$.
To show monotonicity of the \L\ estimator, we rewrite \eqref{LBgen} to obtain
\begin{equation}
\hat{f}^{(L)}(\rho,\vecv)=  \underline{f}^{(\vecv)}(\rho) + \int_{\rho}^1 \frac{\underline{f}^{(\vecv)}(\rho)-\underline{f}^{(\vecv)}(x)}{x^2} dx\ ,
\end{equation}
which is clearly non-increasing with $\rho$.

We now show that $\hat{f}^{(L)}$  dominates all monotone estimators
(and hence is the unique admissible monotone estimator).  By definition,
a monotone estimator $\hat{f}$ can not exceed
$\lambda_L$ on any outcome, that is, it must satisfy the inequalities $\forall \vecv,\ \forall \rho\in[0,1]$:
\begin{align} 
\rho \hat{f}(\rho,\vecv)+&\int_{\rho}^1 \hat{f}(u,\vecv) d u \leq 
\inf_{\vecz\in S^*(\rho,\vecv)} \int_0^1  \hat{f}(u,\vecz) d u =
\nonumber \\
& \inf_{\vecz\in S^*(\rho,\vecv)} f(\vecz) =
\underline{f}^{(\vecv)}(\rho)\ . \label{monotonedef}
\end{align}
Estimator $\hat{f}^{(L)}$ satisfies \eqref{monotonedef} with equalities.
If there is a monotone estimator $\hat{f}$ which is not equivalent to 
$\hat{f}^{(L)}$, that is, for some $\vecv$, the integral is strictly smaller than the
integral of $\hat{f}^{(L)}$ on some interval $(\rho-\epsilon,\rho)$
($\epsilon>0$ may depend on $\vecv$),
we can
obtain a monotone estimator that strictly dominates $\hat{f}$ by
decreasing the estimate for $u\leq \rho-\epsilon$ and increasing it for
$u>\rho-\epsilon$.  The variance decreases because we decrease the estimate on
higher values and increase on lower values. 
\end{proof}


\ignore{
In \eqref{LBgen} we expressed $\hat{f}^{(L)}$ in terms of
the lower bound function $\underline{f}(\rho,\vecv)$.
If $\underline{f}(\rho,\vecv)$ is replaced by
another monotone non increasing relaxation
$g(\rho,\vecv) \leq \underline{f}(\rho,\vecv)$,
the solution is still a well-defined
monotone estimator but may not be variance$^+$-optimal.
A reason for considering a relaxation is that in some cases
$\underline{f}$ may be harder to work with.
}

 Lastly, we show that  $\hat{f}^{(L)}$  is order-optimal
with respect to the order $\prec$ which prioritizes vectors with
lower $f(\vecv)$:
\begin{theorem}
A $\prec^+$-optimal estimator for $f$
 with respect to the partial order
$$\vecv \prec \vecv' \iff f(\vecv)< f(\vecv')\ $$
must be equivalent to $\hat{f}^{(L)}$.
\end{theorem}
\begin{proof}
We use our results of order-optimality (Section \ref{estPREC:sec}).
We can check that we obtain \eqref{LBdefeq} using
\eqref{shscpreccond:eq} and $\prec$ as defined in the statement of the
Theorem.   Thus, a $\prec^+$-optimal
solution must have this form.  
\end{proof}

The \L\ estimator may not be bounded (see Example \ref{example4}).
An estimator that is both bounded and competitive (but not necessarily
in-range, not monotone, and has a large compettive ratio) is the J
estimator \cite{CKsharedseed:2012}.  

\section{Order-optimality}  \label{estPREC:sec}

We identify conditions on $f$ and $\prec$ under which a $\prec^+$-optimal
estimator exists and specify this estimator as a solution of a set of equations.
Our derivations of
$\prec^+$-optimal estimators follow the intuition to require
the estimate on an outcome $S$ to be
$\vecv$-optimal with respect to the $\prec$-minimal vector that is
consistent with the outcome:
\begin{equation} \label{pplusoptd}
\forall S=S(\rho,\vecv),\ \hat{f}(S) = \lambda(\rho,\min_{\prec}(S^*)\ .
\end{equation}

When $\prec$ is a total order and
$V$ is finite, $\min_{\prec}(S^*)$ is unique and \eqref{pplusoptd}
is well defined.
 Moreover, as long as $f$ has a nonnegative unbiased estimator, a solution \eqref{pplusoptd}
always exists and is $\prec^+$-optimal.  We preview a simple construction of the solution:
 Process vectors in increasing $\prec$
order, iteratively building a partially
defined nonnegative estimator.  When processing $\vecv$, the
estimator is already defined for $S(u,\vecv)$ for $u\geq \rho_v$,
for some $\rho_v\in (0,1]$.  We extend it to the outcomes $S(u,\vecv)$
for $u\leq \rho_v$ using the $\vecv$-optimal
extension $\hat{f}^{(\vecv,\rho_v,M)}(u)$, where $M=\int_{\rho_v}^1
\hat{f}(u,\vecv)du$  (see Theorem~\ref{voptlh}).

We now formulate conditions that will allow us to establish
$\prec^+$-optimality  of a solution of \eqref{pplusoptd}  in more
general settings.
These conditions always hold when $\prec$ is a total order and
$V$ is finite.
Generally, $$\min_{\prec}(S^*)=\{\vecz\in S^* \vert \neg\exists \vecw\in S^*,\ \vecw\prec \vecz \}$$ is a {\em set}  and
\eqref{pplusoptd} is well defined  when
$\forall S$, this set is not empty and $\lambda(\rho,\min_{\prec}(S^*))$ is unique,
that is, the value $\lambda(\rho,\vecz)$ is the same for all $\prec$-minimal vectors $\vecz\in
\min_{\prec}(S^*)$.  A sufficient condition for this is that
\begin{eqnarray}
\lefteqn{  \forall \rho\ \forall \vecv\ \forall x\in
  (0,\underline{f}(\rho,\vecv)] \ \forall \vecz,\vecw\in
 \min_{\prec}(S^*(\rho,\vecv)),}\nonumber\\
&& \inf_{\eta<\rho} \frac{\underline{f}(\eta,\vecz) - x}{\rho-\eta}=
\inf_{\eta<\rho} \frac{\underline{f}(\eta,\vecw) - x}{\rho-\eta}\label{unilambdacond}
 \end{eqnarray}

  In this case, the respective Equation
\eqref{pplusoptd} on $u\in (0,\rho]$ are the same for all $\vecz\in
\min_{\prec}(S^*)$ and thus so are the estimate values
$\hat{f}(u,\vecz)$.


We  say that $Z\subset V$ is {\em $\prec$-bounded}  if
\begin{equation} \label{precbound}
\forall \vecv\in Z\  \exists \vecz\in \min_\prec(Z),\ \vecz\preceq
\vecv
\end{equation}
That is, for all $\vecz\in Z$,
$\vecz$ is $\prec$-minimal or is preceded by some vector that is
$\prec$-minimal in $Z$.

We  say that an outcome $S$ is $\prec$-bounded if $S^*$ is
$\prec$-bounded, that is,
\begin{equation} \label{meanprec}
\forall \vecv\in S^*\  \exists \vecz\in \min_\prec(S^*),\ \vecz\preceq
\vecv
\end{equation}


When all outcomes $S(u,\vecv)$ are $\prec$-bounded,
we say that a set of vectors $R$ {\em represents} $\vecv$ if
any outcome consistent with $\vecv$
has a $\prec$-minimal vector
in $R$:
$$\forall u\in (0,1], \exists \vecz\in R,\
 \vecz\in \min_{\prec}(S^*(u,\vecv))\ .$$

We now show that we can obtain a $\prec^{+}$-optimal estimator
if every vector $\vecv$ has a set of finite size that represents it.
Example \ref{precex} walks through a derivation of
$\prec^+$-optimal  estimators.

\begin{lemma}  \label{discssprecopt}
If $f$ satisfies \eqref{nec_req}, \eqref{unilambdacond},
\eqref{meanprec} and
{\small
\begin{equation*}
\forall \vecv,\ \min \{|R|\, \vert\, \forall u\in (0,1], \exists \vecz\in R,\ \vecz\in \min_{\prec}  S^*(u,\vecv)\} < \infty\ ,\label{finprep}
\end{equation*}
}
then a $\prec^+$-optimal estimator exists and must be
equivalent to a solution of
\eqref{pplusoptd}.
\end{lemma}
\begin{proof}
We provide an explicit
construction of a $\prec^{+}$-optimal estimator for $f$.

Fixing $\vecv$,
we select a finite set of representatives.
We can map the representatives (or a subset of them)  to distinct
subintervals covering $(0,1]$.
The subintervals have the form $(a_i,a_{i-1}]$ where $0= a_n<\cdots a_1<a_0=1$ such that
a representative $\vecz$ that is minimal for $(a_i,a_{i-1}]$ is not
minimal for $u\leq a_i$. Such a mapping  can always be obtained since
from \eqref{openset:lemma},
each vector is consistent with an open interval of the form $(a,1]$, and thus if
$\prec$-minimum at $S^*(u,\vecv)$ (we must have $u>a$) it must be
$\prec$-minimum for $S^*(x,\vecv)$ for $x\in (a,u]$.  Thus, the
region on which $\vecz$ is in $\min_{\prec} S^*(u,\vecv)$ is open
to the left.  We can always choose a mapping such that the left
boundary of this region corresponds to $a_i$.


Let $\vecz^{(i)}$ ($i\in [n]$) be the
representative mapped to outcomes $S(u,\vecv)$  where $u\in (a_{i},a_{i-1}]$.
Since $S^*(u,\vecv)$ is monotone non-decreasing with $u$,
$i<j$ implies that $\vecz^{(i)}\prec \vecz^{(j)}$ or that
they are incomparable in the partial order.

We construct a partially specified nonnegative
estimator in steps, by
solving \eqref{pplusoptd} iteratively for the vectors
$\vecz^{(i)}$.
Initially we invoke Theorem~\ref{voptlh}  to obtain estimate
values for $S(u,\vecz^{(1)})$ $u\in (0,1]$ that minimize the
variance for $\vecz^{(1)}$.  The result is a partially specified
nonnegative estimator.  In particular for $\vecv$, the estimator is
now specified for outcomes $S(u,\vecv)$ where $u\in (a_1,1]$.
Any modification of this estimator on a subinterval of $(a_1,1]$ with
positive measure will strictly increase the variance for $\vecz^{(1)}$ (or result in an estimator that can not be completed to a
nonnegative unbiased one).

 After step $i$, we have a partially specified nonnegative estimator
 that is specified for
$S(u,\vecv)$ for $u\in (a_i,1]$.  The estimator is fully specified for
$\vecz^{(j)}$ $j\leq i$ and is $\prec^+$-optimal on these vectors in the sense that
any other partially specified nonnegative estimator that is fully
specified for $\vecz^{(j)}$ $j\leq i$ and has strictly lower variance on some
$\vecz^{(j)}$ ($j\leq i$) must have strictly higher variance on some $\vecz^{(h)}$
such that $h<j$.

We now invoke Theorem~\ref{voptlh} with respect to the vector $\vecz^{(i+1)}$.  The estimator is partially specified for $S(u,\vecz^{(i+1)})$ on $u>a_{i}$ and we obtain estimate values for
the outcomes $S(u,\vecz^{(i+1)})$ for
$u\in (0,a_{i}]$ that constitute a partially specified nonnegative
estimator with minimum variance for $\vecz^{(i+1)}$.
Note again that this completion is unique (up to equivalence).
This extension now defines
$S(u,\vecv)$ for $u\in (a_{i+1},1]$.

Lastly, note that
we must have $f(\vecz^{(n)})=f(\vecv)$ because
$f(\vecz^{(n)})<f(\vecv)$ implies that
\eqref{nec_req} is violated for $\vecv$ whereas the reverse
inequality
implies that
\eqref{nec_req} is violated for $\vecz^{(n)}$.
Since at step $n$ the estimator is specified for all outcomes $S(u,\vecz^{(n)})$ and unbiased, it is unbiased for $\vecv$.

The estimator is
invariant to the choice of the representative sets $R_{\vecv}$ for $\vecv
\in V$ and also remains the same if we restrict $\prec$
so that it includes only relations between $\vecv$ and $R_{\vecv}$.

We so far showed that there is a unique, up to equivalence, partially specified nonnegative
estimator that is $\prec^+$ optimal with respect to a vector $\vecv$
and all vectors it depends on.
Consider now all outcomes $S(u,\vecv)$, for all $u$ and $\vecv$,
arranged according to the containment order on $S^*(u,\vecv)$ according to
decreasing $u$ values with branching points when $S^*(u,\vecv)$
changes.
If for two vectors $\vecv$ and $\vecz$, the sets of outcomes
$S(u,\vecv), u\in (0,1]$ and $S(u,\vecz), u\in (0,1]$ intersect,
the intersection must be equal for  $u>\rho$ for some $\rho<1$.
In this case the estimator values computed with respect to either
$\vecz$ or $\vecv$ would be identical for $u\in (\rho,1]$.  Also
note that partially specified nonnegative solutions on different branches are
independent.  Therefore, solutions with respect to
different vectors $\vecv$  can be consistently combined to a fully specified estimator.
\end{proof}

\subsection{Continuous domains}
The assumptions of Lemma~\ref{discssprecopt} may break on continuous domains.
 Firstly, outcomes may not be $\prec$-bounded and in particular, $\min_{\prec}(S^*)$ can be empty even when
$S^*$ is not, resulting in  \eqref{pplusoptd} not being well
defined. Secondly,
even if $\prec$ is a total order, minimum elements do not necessarily
exist and thus \eqref{meanprec} may not hold, and
lastly, there may not be a finite set of representatives.
To treat such domains, we utilize a notion of {\em convergence with respect to
$\prec$}:

We define
the {\em $\prec$-$\lim$} of a function $h$ on a set of
vectors $Z\subset V$:
\begin{eqnarray}
&& \prec\mbox{-}\lim(h(\cdot),Z)=x \iff  \label{precoptdef:alg}\\
&& \forall \vecv\in Z\  \forall \epsilon>0\ \exists \vecw\preceq
    \vecv, \forall \vecz\preceq
\vecw,\ |h(\vecz)- x| \leq \epsilon\nonumber
\end{eqnarray}
The $\prec$-$\lim$ may not exist but is unique if it does.
Note that when $Z$ is finite or more generally, $\prec$-bounded, and
$h(\vecz)$ is unique for
all $\vecz\in \min_{\prec} Z)$, then
${\prec}\text{-}\lim(h(\cdot),Z)=h(\min_{\prec} Z)$.


 We define the $\prec$-closure of $\vecz$ as the set containing $\vecz$ and
all preceding vectors $\cl_{\prec}(\vecz) = \{\vecv\in V | \vecv\preceq \vecz\}$.

We provide  an alternative definition of the $\prec$-$\lim$
using the notion of $\prec$-closure.
\begin{eqnarray}
&& \prec\text{-}\lim(h(\cdot),Z)=x  \label{supinfdefpreclim}\\
 &\iff & \inf_{\vecv\in  Z}\sup_{\vecz\in \cl_\prec(\vecv)\cap Z} h(\vecz)=
 \sup_{\vecv\in Z}\inf_{\vecz\in \cl_\prec(\vecv)\cap Z} h(\vecz)=x\nonumber
\end{eqnarray}


We say that the lower bound function
 {\em$\prec$-converges} on outcome $S=S(\rho,\vecv)$
if ${\prec}\text{-}\lim(\underline{f}(\eta,\cdot),S^*)$ exists
for all $\eta\in (0,\rho)$.
When this holds,
the $\prec\text{-}\lim$ of the optimal values
\eqref{lambdaMdef} over consistent vectors  $S^*$ exists
for all $M=\int_{\rho}^1 \hat{f}(u,\vecv)du \leq
\underline{f}(\rho,\vecv)$.
We use the notation
\begin{eqnarray*}
\lefteqn{\lambda_{\prec}(S,M)={\prec}\text{-}\lim(\lambda(\rho,\cdot,M),S^*)}\\
&=&\inf_{0\leq \eta < \rho} \frac{\prec\text{-}\lim(\underline{f}(\eta,\cdot),S^*)-M}{\rho-\eta}\ .
\end{eqnarray*}
When the partially specified estimator $\hat{f}$ is clear from context,
we omit the parameter $M$ and use the notation
\begin{eqnarray*}
\lefteqn{\lambda_{\prec}(S)={\prec}\text{-}\lim(\lambda(\rho,\cdot),S^*)}\\
&=&\inf_{0\leq \eta < \rho} \frac{\prec\text{-}\lim(\underline{f}(\eta,\cdot),S^*)-\int_{\rho}^1 \hat{f}(u,\vecv)
  du}{\rho-\eta}\ .
\end{eqnarray*}

We can finally propose a generalization of \eqref{pplusoptd}:
\begin{equation}\label{shscpreccond:eq}\forall S,\ \hat{f}(S)=
  \lambda_{\prec}(S)\ \end{equation}
which is well
defined when the lower bound function
$\prec$-converges for all $S$:
\begin{equation} \label{pwconv}
\forall  S=S(\rho,\vecv), \forall \eta\leq \rho,\
{\prec}\text{-}\lim(\underline{f}(\eta,\cdot),S^*)\ \text{exists.}
\end{equation}

\ignore{
In the remaining part of the section we show the following:
\begin{theorem} \label{precpluscon:thm}
If $f$ satisfies \eqref{nec_req}  and  \eqref{pwconv}
then
a $\prec^+$-optimal estimator exists and must be
equivalent to a solution of \eqref{shscpreccond:eq}.
\end{theorem}
}
Using the definition \eqref{supinfdefpreclim} of $\prec$-convergence
and \eqref{lebesguediff} we obtain that an estimator is
equivalent to \eqref{shscpreccond:eq} if and only if
\begin{equation}\label{eshscpreccond:eq}
\forall \vecv\forall \rho\in (0,1],  \lim_{\eta\rightarrow \rho^-} \frac{\int_\eta^{\rho} \hat{f}(u,\vecv)du}{\rho-\eta}  = \lambda_{\prec}(\rho,\vecv)
\end{equation}

We show that equivalence to \eqref{shscpreccond:eq}
is {\em necessary} for $\prec^+$-optimality.  To facilitate the proof,
we express $\prec^+$-optimality in terms of
restricted admissiblity:
\begin{lemma} \label{varpprecp:lemma}
 An  estimator is $\prec^+$-optimal if and only if, for all $\vecv\in
 V$,  it is  admissible with respect to $\cl_\prec(\vecv)$.
\end{lemma}
\begin{proof}
If there is $\vecv$ such that $\hat{f}$ is not admissible
on $\cl_\prec(\vecv)$,
there is an alternative estimator with strictly lower
variance on some $\vecz\in \cl_\prec(\vecv)$ and at most the variance on all
$\cl_\prec(\vecv)\setminus\{\vecz\}$.  Since $\cl_\prec(\vecv)$ contains all vectors that
precede $\vecz$, the estimator $\hat{f}$ can not be
$\prec^+$-optimal.
To establish the converse, assume an estimator $\hat{f}$ is
admissible on $\cl_\prec(\vecv)$ for all $\vecv$.  Consider
$\vecz\in V$.  Since $\hat{f}$ is admissible on
$\cl_\prec(\vecz)$,
there is no alternative estimator with strictly lower variance on
$\vecz$ and at most the variance of $\hat{f}$ on all preceding vectors.  Since
this holds for all $\vecz$, we obtain that $\hat{f}$ is $\prec^+$-optimal.
\end{proof}

\begin{lemma}
If $f$ satisfies \eqref{nec_req}  and  \eqref{pwconv}
then $\hat{f}$ is $\prec^{+}$-optimal only if
it satisfies \eqref{eshscpreccond:eq}.
\end{lemma}
\begin{proof}
Lemma~\ref{varpprecp:lemma} states that an estimator is $\prec^+$-optimal
if and only if $\forall \vecw\in V$ it is admissible with respect to $\cl_\prec(\vecw)$.
Applying Lemma~\ref{rvarplus}, the latter holds only if
\begin{eqnarray}
\lefteqn{ \forall \vecv\in V\, \forall \rho\in (0,1]}  \label{iff1}\\
 \lim_{\eta\rightarrow \rho^-} \frac{\int_\eta^{\rho} \hat{f}(u,\vecv)du}{\rho-\eta} & \geq &\inf_{\vecz\in \cl_\prec(\vecv)\cap S^*(\rho,\vecv)}
 \lambda(\rho,\vecz)  \nonumber \\
& \leq &  \sup_{\vecz\in \cl_\prec(\vecv)\cap
 S^*(\rho,\vecv)} \lambda(\rho,\vecz) \nonumber
\end{eqnarray}
From definition,
$S(\rho,\vecz)\equiv S(\rho,\vecv)$
for all vectors $\vecz\in S^*(\rho,\vecv)$.
Moreover, for $\vecz\in S^*(\rho,\vecv)$ there is a nonempty
interval $(\eta_z,\rho]$ such that $\forall u\in (\eta_z,\rho]$,
$S^*(u,\vecz\equiv S^*(u,\vecv)$.
Therefore,  for all $\vecz\in S^*(\rho,\vecv)$, the limits
$\lim_{\eta\rightarrow \rho^-} \frac{\int_\eta^{\rho} \hat{f}(u,\vecz)du}{\rho-\eta}$ are the same.
Therefore, \eqref{iff1} $\iff$
\begin{eqnarray}
\lefteqn{ \forall \vecv\in V\, \forall \rho\in (0,1]}  \label{iff2}\\
 \lim_{\eta\rightarrow \rho^-} \frac{\int_\eta^{\rho} \hat{f}(u,\vecv)du}{\rho-\eta} & \geq & \sup_{\vecw\in S^*(\rho,\vecv)} \inf_{\vecz\in \cl_\prec(\vecw)\cap S^*(\rho,\vecv)}
 \lambda(\rho,\vecz)  \nonumber \\
& \leq &  \inf_{\vecw\in S^*(\rho,\vecv)} \sup_{\vecz\in \cl_\prec(\vecw)\cap S^*(\rho,\vecv)}
 \lambda(\rho,\vecz)  \nonumber
\end{eqnarray}
\end{proof}

  We leave open the question of determining the most inclusive
  conditions on $f$ and $\prec$ 
under which a
  $\prec^+$-optimum exists, and thus
the solution of \eqref{shscpreccond:eq}
  is $\prec^+$-optimal.  
We  show that any solution of \eqref{shscpreccond:eq}
is unbiased and nonnegative when $f$ has a nonnegative unbiased estimator.
\begin{lemma}
When $f$ and $\prec$ satisfy \eqref{nec_req} and \eqref{pwconv},
a solution $\hat{f}^{(\prec+)}$  of
\eqref{shscpreccond:eq} is unbiased and nonnegative.
\end{lemma}
\begin{proof}
From Lemma~\ref{rangeunon},
since all values are in-range, the solution
is unbiased and nonnegative.
\end{proof}

\section{The \U\ Estimator} \label{Uest:sec}

The estimator
$\hat{f}^{(U)}$  satisfies \eqref{b3:ineq} with equality.
\begin{equation} \label{Udefeq:eq}
\forall S(\rho,\vecv),\ \hat{f}(\rho,\vecv) = \sup_{\vecz\in S^*}\inf_{0\leq \eta < \rho} \frac{\underline{f}(\eta,\vecz)-\int_{\rho}^1 \hat{f}(u,\vecv) du}{\rho-\eta}
 \end{equation}

 The \U\ estimator is not always admissible. We do show, however,
that under a natural condition, it is
order-optimal with respect to an order that prioritizes vectors
with higher $f$ values (and hence also admissible).
 The condition states that
for all $S(\rho,\vecv)$ and $\eta<\rho$, the 
supremum of the lower bound function $\underline{f}(\eta,\vecz)$
over $\vecz\in S^*$ is attained (in the limiting sense) at
vectors that maximize $f$ on $S^*$.
Formally:
\begin{equation}  \label{UGcond}
\forall \eta < \rho,\ \lim_{x\rightarrow \overline{f}(S)} \sup_{\vecz\in S^* | f(\vecz)\geq x} \underline{f}(\eta,\vecz) = \sup_{\vecz\in S^*}
\underline{f}(\eta,\vecz)\ ,
\end{equation}
where $\overline{f}(S)=\sup_{\vecz\in S^*} f(\vecz)$.

\begin{lemma}
If $f$ satisfies  \eqref{UGcond}, then the \U\ estimator
is $\prec^+$-optimal with respect to the order
$\vecz\prec \vecv \iff f(\vecz) > f(\vecv)$.
\end{lemma}
\begin{proof}
We can show that when \eqref{UGcond} holds then \eqref{Udefeq:eq}
is the same as \eqref{shscpreccond:eq}.
\end{proof}

The condition \eqref{UGcond} is satisfied by $\range_p$ and $\range_{p+}$.  
In this case, the
conditions of Lemma~\ref{discssprecopt} are also satisfied and thus
the \U\ estimator is $\prec^+$ optimal.

\ignore{
\begin{lemma}
 When $f$ satisfies\eqref{UGcond}
$\hat{f}^{(U)}$ is
$\prec^+$-optimal
estimator, for $\prec$ satisfying
$\vecz\prec \vecv \iff f(\vecz) > f(\vecv)$.
\end{lemma}
\begin{proof}
When \eqref{UGcond} holds then \eqref{Udefeq:eq}
is the same as \eqref{shscpreccond:eq}.
\end{proof}
}


\section{Conclusion} \label{conclu:sec}


We take  an optimization  approach to the 
derivation of estimators, targeting both 
worst-case and common-case variance. 
We explore this for monotone sampling, deriving novel and powerful estimators. 

 An interesting open question for monotone sampling is bounding 
the universal ratio: What is the lowest ratio we can guarantee on any 
monotone estimation problem for which there is an estimator with 
finite variances ?   Our L$^*$ estimator show that the ratio is at 
most $4$.  On the other end, one can construct examples where the 
ratio is at least 1.4.  We partially address this problem in follow-up work. 
 
 Another natural question is to find efficient constructions of estimators with 
{\em instance optimal} competitive ratio.  This question is 
interesting even in the context of specific functions 
(such as exponentiated range, which facilitates $L_p$ difference 
estimation \cite{sdiff_arxiv:2014}). 


Our general treatment of arbitrary functions 
facilitates the design of automated tools 
which derive estimators according to specifications. 
Our motivating applications of monotone sampling are for the important
special case of coordinated shared-seed 
sampling.  We expect more applications in that domain, but also
believe that the monotone sampling formulation will find further
applications in pattern  recognition, and plan to explore this in
future work.
Beyond monotone sampling, we hope that the foundations we provided can 
lead to a better understanding of other sampling schemes, and better estimators.  In 
particular, the projection on a single item of {\em independent}
(rather than coordinated) PPS or bottom-$k$ samples of instances is 
essentially an extended monotone estimation problem with $r$
independent seeds instead of a single seed.

Lastly,  the relevance of our work for the analysis of massive data 
sets is demonstrated (in follow-up work) on two basic 
problems:

\begin{trivlist}
\item {\bf Estimating $L_p$ difference from sampled data 
\cite{sdiff_arxiv:2014}:}
Extending Example \ref{example4}, we derive closed for expressions 
for 
$\range_p$ estimators and their variances, with focus on $p=1,2$. 
We estimate $L_p$ as the $p$th roots of sums of our \L\ and \U\ estimators for exponentiated range 
functions $\range_p$ ($p>0$).   These estimators, for $L_1$ and $L_2$,
were applied to samples of data sets with different characteristics: IP flow 
records exhibited larger differences between bandwidth usage assumed 
by a flow key (IP source destination pair, port, and protocol) in 
different times.  The surnames dataset (frequencies of surnames in 
published books in different years) had more similar values. 
Accordingly, the \U\ estimator, which is optimized for large differences 
dominated on the IP flow records dataset whereas the \L\ estimator 
dominated on the surnames dataset.  This demonstrates the potential 
value in selecting a custom estimator.   The \L\ estimator, however,
which is competitive (the ratio turns out to be 2.5 for $L_1$ and 2 
for $L_2$), never exceedingly underperformed  the \U\
estimator, whereas the \U\ estimator could perform much worse than the 
\L.   This shows the value of variance competitiveness and selecting 
a competitive estimator when there is no understanding of 
patterns in data. For the $L_1$ and $L_2$ differences, we also computed (via a program)  the optimally 
competitive estimator. 
Prior to our work, there were no good 
estimators for L$_p$ differences over coordinated samples for any 
 $p\not=1$. Only a non-optimal 
estimator was known  for $L_1$ \cite{multiw:VLDB2009} and 
for the special case of 0/1 
values for the related Jaccard coefficient 
\cite{Broder:CPM00,BRODER:sequences97}.  Our 
study demonstrates that we obtain accurate estimates even 
when only a small fraction of entries is sampled. 

\item 
{\bf Sktech-based similarity estimation in social networks 
  \cite{CDFGGW:COSN2013}:}
We applied our \L\
estimator to obtain sketch-based closeness similarities between nodes in social networks. 
  As mentioned in the introduction, 
a set of all-distances sketches (ADS) can be computed for all nodes 
in near-linear time \cite{ECohen6f}.  The ADS of 
a node $v$  is essentially a sample of other nodes, where a node $u$
is included with probability inversely proportional to 
their Dijkstra rank (neighbor rank) with respect to $v$.  
As mentioned, ADSs of  different nodes are coordinated samples. 
 Closeness similarity \cite{CDFGGW:COSN2013}  between nodes measures 
the similarity of their distance relation to other nodes:
$$\text{sim}(u,v) = \frac{\sum_i \alpha(\max\{d_{vi},d_{ui}\})}{\sum_i 
\alpha(\min\{d_{vi},d_{ui}\})}\ ,$$ where $\alpha$ is non-increasing. 
To estimate closeness similarity of $u$ and $v$ from ADS$(u)$ and 
ADS$(v)$,
we use the HIP inclusion probabilities (which with conditioning allow 
us to consider one item at a time) \cite{ECohenADS:PODS2014}. We then applied 
the \L\ estimator to estimate, for each node $i$, 
$\alpha(\min\{d_{vi},d_{ui}\})$.  These unbiased nonnegative estimates 
were then added up to obtain an estimate for the sum. 
\end{trivlist}





\subsection*{Acknowledgement}
The author would like to thank Micha Sharir for his help in the proof 
of Lemma~\ref{Lestbound}. 
The author 
is grateful to Haim Kaplan for many comments and helpful feedback.

{\small 
\bibliographystyle{plain}
\bibliography{cycle} 

\begin{thebibliography}{10}

\bibitem{BHRSG:sigmod07}
K.~S. Beyer, P.~J. Haas, B.~Reinwald, Y.~Sismanis, and R.~Gemulla.
\newblock On synopses for distinct-value estimation under multiset operations.
\newblock In {\em SIGMOD}, pages 199--210. ACM, 2007.

\bibitem{BrEaJo:1972}
K.~R.~W. Brewer, L.~J. Early, and S.~F. Joyce.
\newblock Selecting several samples from a single population.
\newblock {\em Australian Journal of Statistics}, 14(3):231--239, 1972.

\bibitem{BRODER:sequences97}
A.~Z. Broder.
\newblock On the resemblance and containment of documents.
\newblock In {\em Proceedings of the Compression and Complexity of Sequences},
  pages 21--29. IEEE, 1997.

\bibitem{Broder:CPM00}
A.~Z. Broder.
\newblock Identifying and filtering near-duplicate documents.
\newblock In {\em Proc.of the 11th Annual Symposium on Combinatorial Pattern
  Matching}, volume 1848 of {\em LNCS}, pages 1--10. Springer, 2000.

\bibitem{BCMS:ton2004}
J.~W. Byers, J.~Considine, M.~Mitzenmacher, and S.~Rost.
\newblock Informed content delivery across adaptive overlay networks.
\newblock {\em IEEE/ACM Trans. Netw.}, 12(5):767--780, October 2004.

\bibitem{ECohen6f}
E.~Cohen.
\newblock Size-estimation framework with applications to transitive closure and
  reachability.
\newblock {\em J. Comput. System Sci.}, 55:441--453, 1997.

\bibitem{sdiff_arxiv:2014}
E.~Cohen.
\newblock Distance queries from sampled data: Accurate and efficient.
\newblock Technical Report cs.DS/1203.4903, arXiv, 2012.

\bibitem{ECohenADS:PODS2014}
E.~Cohen.
\newblock All-distances sketches, revisited: Hip estimators for massive graphs
  analysis.
\newblock In {\em PODS}. ACM, 2014.

\bibitem{CDFGGW:COSN2013}
E.~Cohen, D.~Delling, F.~Fuchs, A.~Goldberg, M.~Goldszmidt, and R.~Werneck.
\newblock Scalable similarity estimation in social networks: Closeness, node
  labels, and random edge lengths.
\newblock In {\em COSN}, 2013.

\bibitem{CoKa:jcss07}
E.~Cohen and H.~Kaplan.
\newblock Spatially-decaying aggregation over a network: model and algorithms.
\newblock {\em J. Comput. System Sci.}, 73:265--288, 2007.
\newblock Full version of a SIGMOD 2004 paper.

\bibitem{bottomk07:ds}
E.~Cohen and H.~Kaplan.
\newblock Summarizing data using bottom-k sketches.
\newblock In {\em Proceedings of the ACM PODC'07 Conference}, 2007.

\bibitem{bottomk:VLDB2008}
E.~Cohen and H.~Kaplan.
\newblock Tighter estimation using bottom-k sketches.
\newblock In {\em Proceedings of the 34th VLDB Conference}, 2008.

\bibitem{CK:sigmetrics09}
E.~Cohen and H.~Kaplan.
\newblock Leveraging discarded samples for tighter estimation of multiple-set
  aggregates.
\newblock In {\em Proceedings of the ACM SIGMETRICS'09 Conference}, 2009.

\bibitem{CK:pods11}
E.~Cohen and H.~Kaplan.
\newblock Get the most out of your sample: Optimal unbiased estimators using
  partial information.
\newblock In {\em Proc. of the 2011 ACM Symp. on Principles of Database Systems
  (PODS 2011)}. ACM, 2011.
\newblock full version: {\tt http://arxiv.org/abs/1203.4903}.

\bibitem{CKsharedseed:2012}
E.~Cohen and H.~Kaplan.
\newblock What you can do with coordinated samples.
\newblock In {\em The 17th. International Workshop on Randomization and
  Computation (RANDOM)}, 2013.
\newblock full version: {\tt http://arxiv.org/abs/1206.5637}.

\bibitem{multiw:VLDB2009}
E.~Cohen, H.~Kaplan, and S.~Sen.
\newblock Coordinated weighted sampling for estimating aggregates over multiple
  weight assignments.
\newblock {\em Proceedings of the VLDB Endowment}, 2(1--2), 2009.
\newblock full version: {\tt http://arxiv.org/abs/0906.4560}.

\bibitem{CoWaSu}
E.~Cohen, Y.-M. Wang, and G.~Suri.
\newblock When piecewise determinism is almost true.
\newblock In {\em {Proc. Pacific Rim International Symposium on Fault-Tolerant
  Systems}}, pages 66--71, December 1995.

\bibitem{DDGR07}
A.~Das, M.~Datar, A.~Garg, and S.~Rajaram.
\newblock Google news personalization: scalable online collaborative filtering.
\newblock In {\em WWW}, 2007.

\bibitem{DLT:jacm07}
N.~Duffield, M.~Thorup, and C.~Lund.
\newblock Priority sampling for estimating arbitrary subset sums.
\newblock {\em J. Assoc. Comput. Mach.}, 54(6), 2007.

\bibitem{ES:IPL2006}
P.~S. Efraimidis and P.~G. Spirakis.
\newblock Weighted random sampling with a reservoir.
\newblock {\em Inf. Process. Lett.}, 97(5):181--185, 2006.

\bibitem{GT:spaa2001}
P.~Gibbons and S.~Tirthapura.
\newblock Estimating simple functions on the union of data streams.
\newblock In {\em Proceedings of the 13th Annual ACM Symposium on Parallel
  Algorithms and Architectures}. ACM, 2001.

\bibitem{Gibbons:vldb2001}
P.~B. Gibbons.
\newblock Distinct sampling for highly-accurate answers to distinct values
  queries and event reports.
\newblock In {\em International Conference on Very Large Databases (VLDB)},
  pages 541--550, 2001.

\bibitem{HYKS:VLDB2008}
M.~Hadjieleftheriou, X.~Yu, N.~Koudas, and D.~Srivastava.
\newblock Hashed samples: Selectivity estimators for set similarity selection
  queries.
\newblock In {\em Proceedings of the 34th VLDB Conference}, 2008.

\bibitem{Hajekbook1981}
J.~H{\'a}jek.
\newblock {\em Sampling from a finite population}.
\newblock Marcel Dekker, New York, 1981.

\bibitem{HT52}
D.~G. Horvitz and D.~J. Thompson.
\newblock A generalization of sampling without replacement from a finite
  universe.
\newblock {\em Journal of the American Statistical Association},
  47(260):663--685, 1952.

\bibitem{Knuth2f}
D.~E. Knuth.
\newblock {\em The Art of Computer Programming, Vol 2, Seminumerical
  Algorithms}.
\newblock Addison-Wesley, 1st edition, 1968.

\bibitem{Lanke:Metrica1973}
J.~Lanke.
\newblock On {UMV}-estimators in survey sampling.
\newblock {\em Metrika}, 20(1):196--202, 1973.

\bibitem{LiChurchHastie:NIPS2008}
P.~Li, , K.~W. Church, and T.~Hastie.
\newblock One sketch for all: Theory and application of conditional random
  sampling.
\newblock In {\em NIPS}, 2008.

\bibitem{MS:PODC06}
D.~Mosk-Aoyama and D.~Shah.
\newblock Computing separable functions via gossip.
\newblock In {\em Proceedings of the ACM PODC'06 Conference}, 2006.

\bibitem{surveysampling2:book}
P.~Mukhopandhyay.
\newblock {\em Theory and Methods of Survey Sampling}.
\newblock PHI learning, New Delhi, 2 edition, 2008.

\bibitem{Ohlsson_SPS:1998}
E.~Ohlsson.
\newblock Sequential poisson sampling.
\newblock {\em J. Official Statistics}, 14(2):149--162, 1998.

\bibitem{Ohlsson:2000}
E.~Ohlsson.
\newblock Coordination of pps samples over time.
\newblock In {\em The 2nd International Conference on Establishment Surveys},
  pages 255--264. American Statistical Association, 2000.

\bibitem{Rosen1972:successive}
B.~Ros{\'e}n.
\newblock Asymptotic theory for successive sampling with varying probabilities
  without replacement, {I}.
\newblock {\em The Annals of Mathematical Statistics}, 43(2):373--397, 1972.

\bibitem{Rosen1997a}
B.~Ros{\'e}n.
\newblock Asymptotic theory for order sampling.
\newblock {\em J. Statistical Planning and Inference}, 62(2):135--158, 1997.

\bibitem{Saavedra:1995}
P.~J. Saavedra.
\newblock Fixed sample size pps approximations with a permanent random number.
\newblock In {\em Proc. of the Section on Survey Research Methods, Alexandria
  VA}, pages 697--700. American Statistical Association, 1995.

\bibitem{Vit85}
J.S. Vitter.
\newblock Random sampling with a reservoir.
\newblock {\em ACM Trans. Math. Softw.}, 11(1):37--57, 1985.

\end{thebibliography}
}


\end{document}